\documentclass[10pt,final]{amsart}
\usepackage{mypreamble}

\usepackage[notref,notcite]{showkeys}
\usepackage[colorlinks, linkcolor=blue, citecolor=black]{hyperref}

\newtheorem{thm}[subsubsection]{Theorem}
\newtheorem{lem}[subsubsection]{Lemma}
\newtheorem{prop}[subsubsection]{Proposition}
\newtheorem{cor}[subsubsection]{Corollary}

\theoremstyle{remark}
\newtheorem{rem}[subsubsection]{Remark}

\theoremstyle{definition}
\newtheorem{defn}[subsubsection]{Definition}
\newtheorem{eg}[subsubsection]{Example}

\numberwithin{equation}{subsubsection}

\nc\twist[2]{{_{#1} #2}}

\begin{document}
\title{The moduli stack of $G$-bundles}
\author{Jonathan Wang}
\maketitle


\tableofcontents 

\section{Introduction} 
Let $X$ be a smooth projective curve over an algebraically closed field $k$ of 
characteristic $0$. Then the moduli stack $\Bun_r$ is defined to be the $k$-stack
whose fiber over a $k$-scheme $T$ consists of the groupoid of vector bundles, or equivalently
locally free sheaves, of rank $r$ on $X \xt T$. 
The geometric Langlands program studies the correspondence between $\cD$-modules 
on $\Bun_r$ and local systems on $X$, and this correspondence has deep connections to
both number theory and quantum physics (see \cite{Frenkel-gauge}). 
It is therefore important to understand 
the moduli problem of vector bundles on families of curves. 
As vector bundles of rank $r$ correspond to $\GL_r$-bundles, we naturally extend our
study to $G$-bundles for an algebraic group $G$. 
Since $G$-bundles may have very large automorphism groups, the natural setting
for understanding their moduli problem is the moduli stack 
(as opposed to the coarse moduli space). 

Let $X \to S$ be a morphism of schemes over an arbitrary 
base field $k$. We focus our attention on the $S$-stack $\Bun_G$ whose fiber over an $S$-scheme 
$T$ consists of the groupoid of principal $G$-bundles over $X \xt_S T$. Artin's 
notion of an algebraic stack provides a nice framework for our study of $\Bun_G$. 
The goal of this paper is to give an expository account 
of the geometric properties of the moduli stack $\Bun_G$. The main theorem concerning 
$\Bun_G$ is the following:

\begin{thm} \label{BunG-alg} 
For a flat, finitely presented, projective morphism 
$X \to S$, the $S$-stack $\Bun_G$ is an algebraic stack 
locally of finite presentation over $S$, with a schematic, affine diagonal of finite presentation. 
Additionally, $\Bun_G$ admits an open covering by algebraic substacks of finite presentation
over $S$. 
\end{thm}

The proof that $\Bun_G$ is algebraic has been known to experts for some time, 
and the result follows from \cite{Olsson-Hom-stack} (see also \cite{Behrend-thesis},
\cite{Broshi}, and \cite{Sorger}). To our knowledge, 
however, there is no complete account in the literature of the proof we present, which 
gives more specific information for the case of an algebraic group $G$. 

We now summarize the contents of this paper. In \S\ref{section:quotient-stacks}, we introduce
the stack quotient $[Z/G]$ of a $k$-scheme $Z$ by a $G$-action, which is a fundamental object for
what follows. We show that these quotients are algebraic stacks, and we discuss 
the relations between two different quotients. In order to study $\Bun_G$, we 
prove several properties concerning Hom stacks in \S\ref{section:hom-stacks}. 
The key result in this section is that the sheaf of sections of a quasi-projective 
morphism $Y \to X$ is representable over $S$. Theorem~\ref{BunG-alg} is proved in 
\S\ref{section:Bun_G-presentation}. In \S\ref{section:level-structure},
we define a level structure on $\Bun_G$ using nilpotent thickenings of an $S$-point of
$X$. This provides an alternative presentation of a quasi-compact open substack of
$\Bun_G$. Lastly in \S\ref{section:smoothness}, we prove that $\Bun_G$ is smooth over $S$ 
if $G$ is smooth over $k$ and $X \to S$ has fibers of dimension $1$.

\subsection{Acknowledgments} 
I would like to thank my advisor, Dennis Gaitsgory, for introducing and teaching 
this beautiful subject to me. My understanding of what is written here results from
his invaluable support and guidance. I also wish to thank Thanos D.~Papa\"ioannou
and Simon Schieder for many fruitful discussions.  

\subsection{Notation and terminology} 
We fix a field $k$, and we will mainly work in the category $\Sch_{/k}$ of schemes over $k$. 
For two $k$-schemes $X$ and $Y$, we will use $X \xt Y$ to denote $X \xt_{\spec k} Y$. 
We will consider a scheme as both a geometric object and a sheaf 
on the site $(\Sch)_{\textrm{fpqc}}$ without distinguishing the two. 
For $k$-schemes $X$ and $S$, we denote $X(S) = \Hom_k(S,X)$. For two
morphisms of schemes $X \to S$ and $Y \to S$, we use $\pr_1 : X \xt_S Y \to X$
and $\pr_2 : X \xt_S Y \to Y$ to denote the projection morphisms when there is no
ambiguity. Suppose we have a morphism of schemes $p:X \to S$ and an $\cO_X$-module $\cF$ 
on $X$. For a morphism $T \to S$, we will use the notation $X_T = X \xt_S T$
and $p_T = \pr_2 : X_T \to T$, which will be clear from the context. 
In the same manner, we denote $\cF_T = \pr_1^*\cF$ 
(this notation also applies when $X=S$). For a point $s \in S$, we let $\kappa(s)$ 
equal the residue field. Then using the previous notation, we call the fiber 
$X_s = X_{\spec \kappa(s)}$ and $\cF_s = \cF_{\spec \kappa(s)}$.
For a locally free $\cO_X$-module $\cE$ of finite rank, 
we denote the dual module $\cE^\vee = \cHom_{\cO_X}(\cE,\cO_X)$, which is also
locally free of finite rank. 

We will use Grothendieck's definitions of quasi-projective \cite[5.3]{EGA2} 
and projective \cite[5.5]{EGA2} morphisms. That is, a morphism $X \to S$ of schemes
is quasi-projective if it is of finite type and there exists a relatively ample
invertible sheaf on $X$. A morphism $X \to S$ is projective if there exists
a quasi-coherent $\cO_S$-module $\cE$ of finite type such that $X$ is $S$-isomorphic to
a closed subscheme of $\bP(\cE)$.
We say that $X \to S$ is locally quasi-projective (resp.~locally projective) 
if there exists an open covering $(U_i \subset S)$ such that the morphisms $X_{U_i} \to U_i$ 
are quasi-projective (resp.~projective). 
Following \cite{AK80}, a morphism $X \to S$
is strongly quasi-projective (resp.~projective) if it is finitely presented 
and there exists a locally free $\cO_S$-module $\cE$ of constant finite rank such that
$X$ is $S$-isomorphic to a retrocompact (resp.~closed) subscheme of $\bP(\cE)$. 

We will use $G$ to denote an algebraic group over $k$, by which we mean 
an affine group scheme of finite type over $k$ (note that we do not require $G$
to be reduced). For a $k$-scheme $S$, we say that a sheaf $\cP$ on $(\Sch_{/S})_{\textrm{fpqc}}$
is a right $G$-bundle over $S$ if it is a right $G|_S$-torsor 
(see \cite{De-Ga}, \cite{Giraud} for basic
facts on torsors). Since affine morphism of schemes are effective
under fpqc descent \cite[Theorem 4.33]{FGA-explained}, the $G$-bundle $\cP$ is 
representable by a scheme affine over $S$. 
As $G \to \spec k$ is fppf, $\cP$ is in fact locally trivial in the fppf topology. 
If $G$ is smooth over $k$, then $\cP$ is 
locally trivial in the \'etale topology. 

For a sheaf of groups $\cG$ on a site $\cC$ and $S \in \cC$, we will use
$g \in \cG(S)$ to also denote the automorphism of $\cG|_S$ defined by left multiplication
by $g$. This correspondence gives an isomorphism of sheaves $\cG \simeq \Isom(\cG,\cG)$,
where the latter consists of right $\cG$-equivariant morphisms.

We will use pseudo-functors, instead of fibered categories, in our formal setup of
descent theory. The reasoning is that this makes our definitions more intuitive, and
pseudo-functors are the $2$-category theoretic analogue of presheaves of sets. 
See \cite[3.1.2]{FGA-explained} for the correspondence between pseudo-functors and
fibered categories. By a stack we mean a pseudo-functor satisfying the usual descent
conditions. We will implicitly use the $2$-Yoneda lemma, i.e., for a pseudo-functor 
$\cY$ and a scheme $S$, an element of $\cY(S)$ corresponds to a morphism $S \to \cY$.
We say that a morphism $\cX \to \cY$ of pseudo-functors is representable (resp.~schematic)
if for any scheme $S$ mapping to $\cY$, the $2$-fibered product $\cX \xt_\cY S$
is isomorphic to an algebraic space (resp.~a scheme). 

We use the definition of an algebraic stack given in \cite{stacks-project}:
an algebraic stack $\cX$ over a scheme $S$ is a stack in groupoids on 
$(\Sch_{/S})_{\textrm{fppf}}$ such that the diagonal $\cX \to \cX \xt_S \cX$ is
representable and there exists a scheme $U$ and a smooth surjective morphism $U \to \cX$.
We call such a morphism $U \to \cX$ a presentation of $\cX$.
Note that this definition is weaker than the one in \cite{LMB} as there are less 
conditions on the diagonal. We will use the definitions of properties of 
algebraic stacks and properties of morphisms between algebraic stacks from 
\cite{LMB} and \cite{stacks-project}.

\section{Quotient stacks} \label{section:quotient-stacks}
In this section, fix a $k$-scheme $Z$ with a right $G$-action $\alpha : Z \xt G \to Z$.
All schemes mentioned will be $k$-schemes. 

\begin{defn} \label{defn:quotient-stack} 
The stack quotient $[Z/G]$ is the pseudo-functor $(\Sch_{/k})^\op \to \Gpd$ with
\[ 
[Z/G](S) = \{ \mbox{Right $G$-bundle $\cP \to S$ and $G$-equivariant morphism $\cP \to Z$} \} 
\]
where a morphism from $\cP \to Z$ to $\cP'\to Z$ is a $G$-equivariant morphism 
$\cP \to \cP'$ over $S \xt Z$. 

For $\cdot = \spec k$ with the trivial $G$-action, we call $BG = [\cdot/G]$. 
\end{defn}

\noindent
By considering all of our objects as sheaves on $(\Sch_{/k})_{\textrm{fpqc}}$, we observe 
that $[Z/G]$ is an fpqc stack. The main result of this section is that $[Z/G]$ is
an algebraic stack:

\begin{thm}\label{Z/G-alg} 
The $k$-stack $[Z/G]$ is an algebraic stack with a schematic, separated diagonal. 
If $Z$ is quasi-separated, then the diagonal $\Delta_{[Z/G]}$ is quasi-compact. 
If $Z$ is separated, then $\Delta_{[Z/G]}$ is affine.  
\end{thm}

\begin{rem} \label{rem:general-group-scheme}
In fact, the proof will show that Theorem~\ref{Z/G-alg} holds 
if $G$ is a group scheme affine, flat, and of finite presentation over a base scheme $S$ 
and $Z$ is an $S$-scheme with a right $G$-action, in which case $[Z/G]$ is an algebraic 
$S$-stack. We choose to work over the base $\spec k$ in this section because some results 
in \S\ref{section:cog} do not hold in greater generality.
\end{rem}

\subsection{Characterizing $[Z/G]$} 
It will be useful in the future to know when a stack satisfying some properties is 
isomorphic to $[Z/G]$. In particular, we show the following (the conditions
for a morphism from a scheme to a stack to be a $G$-bundle will be made precise later). 

\begin{lem}\label{Z/G=X}
Suppose $\cY$ is a $k$-stack and $\sigma_0 : Z \to \cY$ is a $G$-bundle. 
Then there is an isomorphism $\cY \to [Z/G]$ of stacks such that the triangle
\[ \xymatrix{ & Z \ar[ld]_{\sigma_0} \ar[rd]^{\alpha : Z \xt G \to Z} 
\\ \cY \ar[rr]^-\sim && [Z/G] } \]
is $2$-commutative.
\end{lem}

\begin{rem}\label{P/G}
Lemma~\ref{Z/G=X} in particular shows that if $X$ is a scheme and
$Z \to X$ is a $G$-bundle over $X$, then $X \simeq [Z/G]$. 
Therefore these two notions of a quotient, 
as a scheme and as a stack, coincide. 
\end{rem}

The rest of this subsection will be slightly technical. 
We first define the notion of a $G$-invariant morphism to a stack,
which then naturally leads to the definition of a morphism being a $G$-bundle. 
We then introduce two lemmas, which we use to prove Lemma~\ref{Z/G=X}.

\subsubsection{$G$-invariant morphisms to a stack} \label{sect:G-invariant}
Let $\cY$ be a $k$-stack and $\sigma_0 : Z \to \cY$ a morphism of stacks. 
We say $\sigma_0$ is \emph{$G$-invariant} if the following conditions are satisfied:
\begin{enumerate}
\item The diagram
\[ \xymatrix{ Z \xt G \ar[r]^-\alpha \ar[d]_{\pr_1} & Z \ar[d] \\ Z \ar[r]
& \cY } \] 
is $2$-commutative. This is equivalent to having a $2$-morphism $\rho : \pr_1^*\sigma_0 
\to \alpha^*\sigma_0$. 
\item For a scheme $S$, if we have $z \in Z(S)$ and $g\in G(S)$, then  
we let $\rho_{z,g}$ denote the corresponding $2$-morphism
\[ z^*\sigma_0 \simeq (z,g)^*\pr_1^*\sigma_0 \overset{(z,g)^*\rho}\lto (z,g)^*\alpha^*\sigma_0
\simeq (z.g)^*\sigma_0. \]
The $\rho_{z,g}$ must satisfy an associativity condition. 
More precisely, for $z \in Z(S)$ and $g_1,g_2 \in G(S)$, we require 
\[ \xymatrix{ z^*\sigma_0 \ar[r]^{\rho_{z,g_1}} \ar[d]_{\rho_{z,g_1g_2}} & 
(z. g_1)^*\sigma_0 \ar[d]^{\rho_{z. g_1,g_2}} \\ 
(z. (g_1g_2))^*\sigma_0 \ar@{=}[r] & ((z. g_1). g_2)^*\sigma_0 } \]
to commute.
\end{enumerate}
For a general treatment of group actions on stacks, we refer the reader to 
\cite{Romagny}. 
\medskip

Suppose that $\sigma_0 : Z \to \cY$ is a $G$-invariant morphism of stacks. 
For any scheme $S$ with a morphism $\sigma : S \to \cY$, the $2$-fibered product $Z \xt_\cY S$ is
a sheaf of sets. We show that $Z \xt_\cY S$ has an induced right $G$-action such that
$Z \xt_\cY S \to S$ is $G$-invariant and $Z \xt_\cY S \to Z$ is $G$-equivariant. 
By definition of the $2$-fibered product, for a scheme $T$ we have 
\begin{equation} \label{eqn:ZxS}
\Bigl(Z \xt_\cY S \Bigr)(T) 
= \bigl\{ (a,b,\phi) \mid a\in Z(T), b \in S(T), \phi : a^*\sigma_0 \to b^*\sigma \bigr\}.
\end{equation}
We define a right $G$-action on $Z \xt_\cY S$ by letting $g \in G(T)$ act via
\[ (a,b,\phi).g = (a.g,b,\phi \circ \rho_{a,g}^{-1}). \]
Since $\rho_{a.g_1,g_2}\circ\rho_{a,g_1} = \rho_{a,g_1g_2}$, this defines a natural 
$G$-action. From this construction, it is evident that $Z\xt_\cY S \to S$ is $G$-invariant
and $Z \xt_\cY S \to Z$ is $G$-equivariant.

\subsubsection{$G$-bundles over a stack} \label{section:G-bundle-stack}

Let $\cY$ be a stack and $\sigma_0 : Z \to \cY$ a $G$-invariant morphism of stacks.
We say that $Z \to \cY$ is a \emph{$G$-bundle} if for any scheme $S$ mapping to $\cY$, 
the induced $G$-action on $Z \xt_\cY S \to S$ gives a $G$-bundle. In particular, $\sigma_0$ 
is schematic. 

\subsubsection{}
Let $\tau_0 \in [Z/G](Z)$ correspond to the trivial $G$-bundle $\pr_1 : Z \xt G \to Z$
with the $G$-equivariant morphism $\alpha : Z \xt G \to Z$. 

\begin{lem}\label{Z/G-atlas} 
The diagram
\[ \xymatrix{ Z \xt G \ar[r]^\alpha \ar[d]_{\pr_1} & Z \ar[d]^{\tau_0} \\ 
Z \ar[r]^{\tau_0} & [Z/G] } \] 
is a $2$-Cartesian square, and $\tau_0$ is a $G$-invariant morphism.  \end{lem}
\begin{proof}
Suppose for a scheme $S$ we have $z,z' \in Z(S)$ and $\phi:z'^*\tau_0 \simeq z^*\tau_0$.
Then $\phi$ corresponds to $g \in G(S)$ such that 
\[ \xymatrix{ S \xt G \ar[rr]^{g} \ar[rd]_{\alpha(z'\xt \id)} && 
S \xt G \ar[ld]^{\alpha(z\xt\id)}  \\ & Z } \]
commutes. This implies that $z. g = z'$. Conversely, $z\in Z(S)$
and $g \in G(S)$ uniquely determine $z' = z. g$ and a $G$-equivariant morphism 
$z'^*\tau_0 \to z^* \tau_0$ corresponding to $g$. Therefore the morphism
\[ (\alpha,\pr_1) : Z \xt G \to Z \xt_{[Z/G]} Z \] 
is an isomorphism.
Taking $S = Z \xt G$ and $\id_{Z \xt G}$ gives an isomorphism
\[ \rho^{-1} : \alpha^*(\tau_0) = (\pr_1.\pr_2)^*(\tau_0) \to  \pr_1^*(\tau_0) ,\] 
which is defined explicitly by $Z \xt G\xt G \to Z \xt G \xt G : (z,g_1,g_2) \mapsto 
(z,g_1,g_1 g_2)$. From this we see that for a scheme $S$ and $z\in Z(S), g\in G(S)$, 
the morphism $\rho_{z,g}$ corresponds to the morphism of schemes $S \xt G \to S \xt G : 
(s,g_0) \mapsto (s, g(s)^{-1} g_0)$. Therefore $g_2^{-1}g_1^{-1} = (g_1g_2)^{-1}$ implies 
$\tau_0$ is a $G$-invariant morphism.
\end{proof}

\begin{rem}\label{atlas-Gsheaf} 
Now that we know $\tau_0$ is $G$-invariant, $Z \xt_{[Z/G]} Z$ has an induced $G$-action
as a sheaf of sets. The morphism $(\alpha,\pr_1) :Z \xt G \to Z\xt_{[Z/G]} Z$ sends $(z,g) 
\mapsto (z.g, z, \rho_{z,g}^{-1})$ on $S$-points, so the associativity condition on $\rho$
implies that $(\alpha,\pr_1)$ is a $G$-equivariant morphism of sheaves of sets.
\end{rem}

\begin{lem}\label{Z/G-triv} 
Let $\tau=(f:\cP \to Z) \in [Z/G](S)$ for a scheme $S$, and suppose $\cP$ 
admits a section $s \in \cP(S)$. Then the $G$-equivariant morphism 
$\wtilde s : S\xt G \to \cP$ induced by $s$ gives an isomorphism $(f\circ s)^*\tau_0 \to 
\tau \in [Z/G](S)$.
\end{lem}
\begin{proof} Let $a = f \circ s : S \to \cP \to Z$. We have a Cartesian square
\[ \xymatrix@C=1.5cm{ S \xt G \ar[r]^-{a \xt \id} \ar[d] & Z \xt G \ar[d]^{\pr_1} \ar[r]^\alpha & 
Z  \\ S \ar[r]^a & Z } \]
so $a^*\tau_0 = (\alpha(a\xt \id) : S\xt G \to Z)$. The diagram 
\[ \xymatrix{ S \xt G \ar[rr]^-{\wtilde s} \ar[rd]_{\alpha (a\xt \id)} && \cP \ar[ld]^f \\ 
& Z } \]
is commutative by $G$-equivariance, since the morphisms agree after composing by the section 
$S \to S \xt G$ corresponding to $1 \in G(S)$. Therefore $\wtilde s$ is a morphism in 
$[Z/G](S)$.
\end{proof}

We are now ready to prove the claimed characterization of $[Z/G]$.

\begin{proof}[Proof of Lemma~\ref{Z/G=X}] 
We define the morphism $F : \cY \to [Z/G]$ by sending $\sigma \in \cY(S)$ to 
\[ F(\sigma) = 
\xymatrix@C=1.5cm{ \ds Z \xt_\cY S \ar[r] \ar[d] & Z \ar@{.>}[d]^{\sigma_0}
\\  S \ar@{.>}[r]^\sigma & \cY } \]
for any scheme $S$. Since $Z \to \cY$ is a $G$-bundle, $Z\xt_\cY S \to Z$ is indeed
an object of $[Z/G](S)$ (see \S\ref{section:G-bundle-stack}). 
For $\sigma, \sigma' \in \cY(S)$, let $\cP = Z \xt_{\cY,\sigma} S$ and 
$\cP' = Z \xt_{\cY,\sigma'} S$ be the fibered products. We have  
$\cP(T) = \{ (a,b,\phi)\}$ following the notation of $\eqref{eqn:ZxS}$. 
A morphism $\psi : \sigma \to \sigma'$ induces 
a $G$-equivariant morphism $\cP \to \cP'$ over $S \xt Z$ by sending 
\begin{equation} \label{eqn:F_S} 
(a,b,\phi) \mapsto (a,b,b^*\psi \circ \phi). 
\end{equation}
This endows $F$ with the structure of a morphism of stacks $\cY \to [Z/G]$.
\smallskip

We will first show $2$-commutativity of the triangle. Then we prove that
$F$ is fully faithful and essentially surjective to deduce that it is an isomorphism. 
\smallskip

\emph{Step 1.} 
Showing the triangle is $2$-commutative is equivalent to giving a morphism $\tau_0 \to F(\sigma_0)$
in $[Z/G](Z)$. The discussion of Remark~\ref{atlas-Gsheaf} shows that 
\[ (\alpha, \pr_1): Z \xt G \to Z \xt_\cY Z\] 
gives a $G$-equivariant morphism of sheaves over $Z \xt Z$, and this gives the desired morphism
$\tau_0 \to F(\sigma_0)$ in $[Z/G](Z)$. 

\smallskip

\emph{Step 2.}
We show that $F$ is fully faithful. Since $\cY, [Z/G]$ are stacks, $F$ is fully faithful if 
and only if for any fixed scheme $S$ and $\sigma, \sigma' \in \cY(S)$, the induced morphism of 
sheaves of sets 
\[ F : \Isom_{\cY(S)}(\sigma,\sigma') \to \Isom_{[Z/G](S)}(F(\sigma),F(\sigma')) \]
on $\Sch_{/S}$ is an isomorphism. 
Let $\cP = Z \xt_{\cY,\sigma} S$ and $\cP'= Z \xt_{\cY,\sigma'} S$. 
By restricting to an fppf covering $(S_i \to S)$, we reduce to the case where 
$\cP, \cP'$ are both trivial bundles.  
Let $\Psi : \cP \to \cP'$ be a $G$-equivariant morphism over $S \xt Z$. Since $\cP$ is trivial, 
there is a section $s \in \cP(S)$ corresponding to $(a, \id_S, \phi : a^*\sigma_0 \to \sigma)$
with $a \in Z(S)$. 
Then $\Psi$ sends $s$ to some element $(a,\id_S, \phi' : a^*\sigma_0 \to \sigma') \in \cP'(S)$. 
We define a morphism of sets
\[ L : \Hom_{[Z/G](S)}(F(\sigma),F(\sigma')) \to \Hom_{\cY(S)}(\sigma,\sigma').\]
by $L(\Psi) = \phi' \circ \phi^{-1} : \sigma \to \sigma'$.

It remains to check that $F,L$ are mutually inverse (here by $F$ we really mean the morphism 
$F_S$ of Hom sets). Starting with 
$\psi : \sigma \simeq \sigma'$, we have from \eqref{eqn:F_S} that 
$F(\psi)$ sends $s$ to $(a,\id_S, \psi \circ \phi) \in \cP'(S)$. 
Thus $LF(\psi) = (\psi \circ \phi) \circ \phi^{-1} =\psi$. If we instead start with 
$\Psi : \cP \to \cP'$, then $L(\Psi) = \phi' \circ \phi^{-1}$ induces 
\[ FL(\Psi) : s \mapsto (a,\id_S, (\phi' \circ \phi^{-1}) \circ \phi) \in \cP'(S) \]
again by \eqref{eqn:F_S}. A $G$-equivariant morphism of trivial bundles $\cP \to \cP'$ over $S$ is 
determined by the image of the section $s \in \cP(S)$; hence $\Psi = FL(\Psi)$. 
We conclude that $F$ is fully faithful.

\smallskip

\emph{Step 3.} 
We now prove that $F$ is essentially surjective and hence an isomorphism of stacks. 
Let $\tau = (f: \cP \to Z) \in [Z/G](S)$. Choose an fppf covering $(j_i : S_i \to S)$ 
trivializing $\cP$. Then we have sections $s_i \in \cP(S_i)$. Let $f_i$ denote the restriction
of $f$ to $\cP_{S_i}$. Lemma~\ref{Z/G-triv} gives an isomorphism 
$(f_i \circ s_i)^*\tau_0 \simeq j_i^*\tau$. We already know that $\tau_0 \simeq F(\sigma_0)$
from Step 1.
Therefore \[ F((f_i\circ s_i)^*\sigma_0) \simeq (f_i \circ s_i)^*F(\sigma_0) \simeq 
(f_i\circ s_i)^*\tau_0 \simeq j_i^*\tau. \]
By \cite[Lemma \href{http://math.columbia.edu/algebraic_geometry/stacks-git/locate.php?tag=046N}{046N}]{stacks-project}, we conclude that $F$ is an isomorphism. 
\end{proof}

\subsection{Twisting by torsors} In this subsection, we review a construction that will
come up in many places. 
Let $\cC$ be a subcanonical site 
with a terminal object and a sheaf of groups $\cG$. 
For $S \in \cC$, let $\cP$ be a right $\cG|_S$-torsor over $S$. Suppose we have
a sheaf of sets $\cF$ on $\cC$ with a left $\cG$-action. Then $\cG|_S$ acts on 
$\cP \xt \cF$ from the right by $(p,z).g = (p.g,g^{-1}.z)$. We have a presheaf $\cQ$
on $\cC/S$ defined by taking $\cG$-orbits $\cQ(U) = (\cP(U) \xt \cF(U)) / \cG(U)$.
We define the sheaf 
\[ \twist \cP \cF = (\cP \xt \cF)/\cG = \cP \xt^\cG \cF \]
obtained by twisting $\cF$ by $\cP$ to be the sheafification of $\cQ$. 

Since sheaves on $\cC$ form a stack, we give an alternative description of
$\twist \cP \cF$ by providing a descent datum. Let $(S_i \to S)$ be a covering
such that $\cP|_{S_i} \simeq \cG|_{S_i}$. Then $(\cG|_{S_i}, g_{ij})$
give a descent datum of $\cP$, for some $g_{ij} \in \cG(S_i \xt_S S_j)$. Note that 
\[ (\cP|_{S_i} \xt \cF)/\cG \simeq (\cG|_{S_i} \xt \cF)/\cG \simeq \cF|_{S_i}, \]
and $\cQ|_{S_i} \simeq \cF|_{S_i}$ is already a sheaf. By the definition of the 
group action on $\cP \xt \cF$, we see that the transition morphism 
$\varphi_{ij} : (\cF|_{S_j})|_{S_i \xt_S S_j}
\to (\cF|_{S_i})|_{S_i \xt_S S_j}$ is given by the left action of $g_{ij}$. 
Since sheafification commutes with the restriction $\cC/S_i \to \cC/S$, we conclude that
\[ (\cF|_{S_i}, \varphi_{ij}) \]
gives a descent datum for $\twist \cP \cF$. 

If $\cF$ is instead a sheaf with a right $\cG$-action, we will use $\twist \cP \cF$ to
denote the twist of $\cF$ considered with the inverse left $\cG$-action.
\medskip

We will use the above twisting construction for the big site $\cC = (\Sch_{/k})_{\textrm{fppf}}$. 
Let $V$ be a $G$-representation. Then $V$ can be considered 
as an abelian fppf sheaf via pullback (see \cite[Lemma \href{http://math.columbia.edu/algebraic_geometry/stacks-git/locate.php?tag=03DT}{03DT}]{stacks-project})
with a left $G$-action. Then for a right $G$-bundle $\cP$ over $S$, 
we have a quasi-coherent sheaf $\twist \cP V$ on $S$ by descent 
\cite[Theorem 4.2.3]{FGA-explained}.

For a $k$-scheme $Y$ with a left or right $G$-action and $\cP$ a right $G$-bundle over $S$, we
can form the \emph{associated fiber bundle} $\twist \cP Y$ over $S$ and 
ask when it is representable 
by a scheme. This will be a key topic in the next subsection.

\subsection{Change of space} 
Let $\beta : Z'\to Z$ be a $G$-equivariant morphism of schemes with right $G$-action. 
Then there is a natural morphism of stacks $[Z'/G] \to [Z/G]$ defined by sending 
\[ (\cP \to Z') \mapsto (\cP \to Z' \oarrow \beta Z).\] 
The next lemma shows that under certain conditions, this morphism is schematic.

\begin{lem} \label{cos-representable} 
The morphism $[Z'/G] \to [Z/G]$ is representable. 
If the morphism of schemes $Z' \to Z$ is affine 
(resp.~quasi-projective with a $G$-equivariant relatively ample invertible sheaf), then 
the morphism of quotient stacks is schematic and affine (resp.~quasi-projective). 
If $Z' \to Z$ has a property that is fppf local on the target, then
so does the morphism of quotient stacks.
\end{lem} 

First, we need a formal result on $2$-fibered products of quotient stacks. 

\begin{lem} \label{cos-fiber}
Let $\beta_i : Z_i \to Z$ be $G$-equivariant morphisms of
schemes with right $G$-action for $i=1,2$. Then the square 
\[ \xymatrix{ \ds \Bigl[ \bigl(Z_1 \xt_Z Z_2\bigr) / G \Bigr] \ar[r] \ar[d] & [Z_1/G] \ar[d] \\ 
[Z_2/G] \ar[r] & [Z/G] } \]
induced by the $G$-equivariant projections $Z_1 \xt_Z Z_2 \to Z_i$ is $2$-Cartesian. 
\end{lem}
\begin{proof}
For a scheme $S$, we have 
\[  \Bigl([Z_1/G] \xt_{[Z/G]} [Z_2/G] \Bigr)(S) = \left\{ \mbox{ 
\begin{tabular}{c} $f_1:\cP_1 \to Z_1$, $f_2: \cP_2 \to Z_2$ \\
$G$-equivariant morphism $\phi : \cP_1 \to \cP_2$ over $S \xt Z$ \end{tabular} } \right\}. \]
Observe that we have a morphism from 
$(\cP_1, f_1,\cP_2, f_2, \phi)$ to $(\cP_1, f_1, \cP_1, f_2 \circ \phi, \id_{\cP_1})$ 
via $\id_{\cP_1}, \phi^{-1}$. The latter is the image of $(\cP_1, 
(f_1,f_2\circ \phi) : \cP_1 \to Z_1 \xt_Z Z_2) \in [(Z_1 \xt_Z Z_2)/G](S)$.  
Hence the functor
\[ F: \Bigl[ \bigl(Z_1 \xt_Z Z_2\bigr)/G \Bigr](S) \to \Bigl([Z_1/G] \xt_{[Z/G]} [Z_2/G]\Bigr)(S) \]
is essentially surjective. 

Take $(f:\cP \to Z_1 \xt_Z Z_2) \in [(Z_1 \xt_Z Z_2)/G](S)$ and let $f_1 : \cP \to Z_1, 
f_2 : \cP \to Z_2$ be the morphisms corresponding to $f$. Then the image of $(\cP,f)$ under $F$ is 
$\tau = (\cP,f_1,\cP,f_2,\id_\cP)$. We similarly associate to $(\cP',f') \in [(Z_1\xt_Z Z_2)/G](S)$
the corresponding $f'_1, f'_2$ and $\tau'$. A morphism $\tau \to \tau'$ is then 
a pair of $G$-equivariant morphisms $\gamma_1 : \cP \to \cP'$ over $S \xt Z_1$ and 
$\gamma_2 : \cP \to \cP'$ over $S \xt Z_2$ intertwining $\id_\cP$ and $\id_{\cP'}$. 
The last condition implies $\gamma_1 = \gamma_2$. Thus $\gamma_1$ is a $G$-equivariant
morphism over $S \xt (Z_1 \xt_Z Z_2)$, i.e., a morphism $(\cP,f) \to (\cP',f')$. We conclude that
$F$ is fully faithful and hence an isomorphism. 
\end{proof}

\begin{proof}[Proof of Lemma~\ref{cos-representable}]
For a scheme $S$, take $(f:\cP \to Z) \in [Z/G](S)$. Remark~\ref{P/G} gives an 
isomorphism $S \simeq [\cP/G]$. Hence by Lemma~\ref{cos-fiber} we have a Cartesian square
\[ \xymatrix{ \ds \Bigl[ \bigl(Z' \xt_Z \cP \bigr)/G \Bigr] \ar[r] \ar[d] & [Z'/G] \ar[d] \\ 
S \ar[r]^{\cP \to Z} &  [Z/G]}\]
We prove that $[(Z' \xt_Z \cP)/G]$ is representable by an algebraic space. 
Remark~\ref{P/G} implies that $[\cP/G](T)$ is an equivalence relation for any scheme $T$. 
It follows that $[(Z' \xt_Z \cP)/G](T)$ is an equivalence relation, so $[(Z' \xt_Z \cP)/G]$
is isomorphic to a sheaf of sets. 
Pick an fppf covering $(S_i \to S)$ so that each $\cP_{S_i}$ admits 
a section $s_i \in \cP(S_i)$. Note that we have a morphism $S_i \xt G \to \cP_{S_i} \to Z$. 
There is a $G$-equivariant isomorphism 
\[ \gamma_i : \Bigl(Z' \xt_{Z, f\circ s_i} S_i \Bigr) \xt G \to Z' \xt_Z (S_i \xt G)  \] 
over $S_i$ where on the left hand side $G$ acts as a trivial bundle over 
$Z' \xt_{Z,f\circ s_i} S_i$, while on the right
hand side $G$ acts on $Z'$ and $S_i \xt G$. The morphism is 
defined by sending $(a,g) \mapsto (a.g,g)$ for $a \in Z'(T), g \in G(T)$ and $T$ an $S_i$-scheme. 
Applying Remark~\ref{P/G} again, we have an isomorphism 
\[ \Phi_i :
Z' \xt_{Z,f\circ s_i} S_i \simeq \Bigl[ \bigl(Z' \xt_{Z,f\circ s_i} S_i \bigr) \xt G/G\Bigr] 
\overset{\gamma_i}\lto
\Bigl[ \bigl(Z' \xt_Z (S_i \xt G) \bigr)/G\Bigr] \overset{\id \xt \wtilde{s_i}}\lto 
\Bigl[ \bigl(Z' \xt_Z \cP_{S_i} \bigr)/G\Bigr] \]
over $S_i$
defined explicitly on $T$-points by sending $a \in (Z' \xt_{Z,f\circ s_i} S_i)(T)$ to  
\[ \xymatrix{ T \xt G \ar[r]^-{a\xt \id} \ar[d] & \ds \Bigl(Z' \xt_{Z,f\circ s_i} S_i\Bigr) \xt G 
\ar[r]^-{\gamma_i} & \ds Z'\xt_Z (S_i \xt G) \ar[r]^-{\id \xt \wtilde{s_i}} & \ds Z' \xt_Z \cP 
\\ T } \]
Therefore after restricting to an fppf covering, $[(Z' \xt_Z \cP)/G] \xt_S S_i 
\simeq [(Z' \xt_Z \cP_{S_i})/G]$ becomes
representable by a scheme. Since algebraic spaces satisfy descent \cite[Lemma \href{http://math.columbia.edu/algebraic_geometry/stacks-git/locate.php?tag=04SK}{04SK}]{stacks-project}, we deduce that
$[(Z' \xt_Z \cP)/G]$ is representable by an algebraic space. 
Before proving when $[(Z' \xt_Z \cP)/G]$ is representable by a scheme, we give
a description of its descent datum. 
\medskip

{\it Descent datum of $[(Z' \xt_Z \cP)/G]$.}
Let $S_{ij} = S_i \xt_S S_j$. There are $g_{ij} \in G(S_{ij})$ such that
$s_j = s_i. g_{ij} \in \cP(S_{ij})$ (we slightly abuse notation and omit restriction signs
when the context is clear). Then the action of $g_{ij}^{-1}$ on $Z'$ induces an isomorphism
\[ \varphi_{ij} : Z' \xt_{Z, f\circ s_j} S_{ij} \to Z' \xt_{Z, (f\circ s_j).g_{ij}^{-1}} S_{ij} 
= Z' \xt_{Z, f\circ s_i} S_{ij}.
\]
It follows from the definitions that for $a \in (Z' \xt_{Z,f\circ s_j} S_{ij})(T)$, the square
\[ \xymatrix@C=4cm{ T \xt G \ar[d]_{g_{ij}} \ar[r]^{(\id \xt \wtilde{s_j}) \circ 
\gamma_j \circ (a \xt \id)} & Z' \xt_Z \cP_{S_{ij}} \ar@{=}[d] \\
T \xt G \ar[r]^{(\id \xt \wtilde{s_i}) \circ \gamma_i \circ (\varphi_{ij} \circ (a \xt \id))} & 
Z' \xt_Z \cP_{S_{ij}} } \]
is commutative. Therefore $\Phi_j(a)$ and $(\Phi_i \circ \varphi_{ij})(a)$ are isomorphic
in $[(Z' \xt_Z \cP_{S_{ij}})/G]$. We conclude that $(Z' \xt_{Z,f\circ s_i} S_i , \varphi_{ij})$
is a descent datum of $[(Z'\xt_Z \cP)/G]$ with respect to the covering $(S_i \to S)$.
\medskip

If $Z' \to Z$ is affine, then so is $Z' \xt_{Z, f\circ s_i} S_i \to S_i$. 
Affine morphisms are effective under descent \cite[Theorem 4.33]{FGA-explained}, so 
in this case $[(Z' \xt_Z \cP)/G]$ is representable by a scheme affine over $S$. 
Next suppose that $Z' \to Z$ is quasi-projective with a $G$-equivariant relatively ample
invertible $\cO_{Z'}$-module $\cL$. 
From our description of the descent datum of $[(Z' \xt_Z \cP)/G]$, 
we have a commutative square 
\[ \xymatrix { 
\ds Z' \xt_{Z, f\circ s_j} S_{ij} \ar[r]^-{\pr_{1,j}} \ar[d]_-{\varphi_{ij}}
& Z' \ar[d]^{g_{ij}^{-1}} \\ 
\ds Z' \xt_{Z, f\circ s_i} S_{ij} \ar[r]^-{\pr_{1,i}} & Z' } \] 
Therefore the $G$-equivariant structure of $\cL$ gives an isomorphism 
\[ \varphi_{ij}^* \pr_{1,i}^* \cL \simeq \pr_{1,j}^* \cL \]
of invertible sheaves relatively ample over $S_{ij}$. Since the $g_{ij}$ satisfy the cocycle
condition, the associativity property of $G$-equivariance implies that the above
isomorphisms satisfy the corresponding cocycle condition. We therefore have a descent
datum of relatively ample invertible sheaves. By descent \cite[VIII, Proposition 7.8]{SGA1}, 
we conclude that $[(Z' \xt_Z \cP)/G]$ is representable by a scheme quasi-projective 
over $S$. 
\end{proof}

\begin{cor} \label{Z/G-BG} 
For a $k$-scheme $Z$ with a right $G$-action and a right $G$-bundle $\cP$ over a $k$-scheme $S$, 
there is an isomorphism $\twist \cP Z \simeq [(Z \xt \cP)/G]$ over $S$. 
\end{cor} 
\begin{proof} 
The proof of Lemma~\ref{cos-representable} shows that $[(Z \xt\cP)/G]$ has a descent datum 
$(Z \xt S_i, \varphi_{ij} )$. This is the same descent datum as $\twist \cP Z$, proving the 
claim. This additionally implies that
\[ \xymatrix{ \twist \cP Z \ar[r] \ar[d] & [Z/G] \ar[d] \\ S \ar[r]^\cP & BG } \]
is a Cartesian square.
\end{proof}

\begin{cor}\label{Z-Z/G}
Given $\tau = (p : \cP \to S, f : \cP \to Z) \in [Z/G](S)$ for a $k$-scheme $S$, 
\[ \xymatrix{ \cP \ar[r]^f \ar[d]_p & Z \ar[d]^{\tau_0} \\ S \ar[r]^{\tau} & [Z/G] } \]
is a Cartesian square. In particular, $Z \to [Z/G]$ is schematic, affine, and fppf. 
\end{cor}
\begin{proof} By Lemma~\ref{cos-fiber}, we have a Cartesian square
\begin{equation} \label{eqn:Z-Z/G:fiber} 
\xymatrix{ \ds \Bigl[ \bigl((Z \xt G) \xt_{\alpha, Z, f} \cP \bigr)/G \Bigr] 
\ar[r] \ar[d] & [Z \xt G/G] \ar[d]^\alpha \\
[\cP/G] \ar[r] & [Z/G] } 
\end{equation}
where $Z \xt G$ is the trivial $G$-bundle over $Z$. 
There is a $G$-equivariant morphism \[ (f \xt \id, \alpha_\cP) : \cP \xt G \to (Z \xt G) 
\xt_{\alpha,Z,f} \cP\] where $\cP \xt G$ is the trivial $G$-bundle over $\cP$.
Action and projection induce a $G$-equivariant morphism $\cP \xt G \to \cP \xt_S \cP$ 
over $\cP \xt \cP$. Therefore we see that the diagram
\[ \xymatrix{ S \ar[dd]_{\id:\cP \to \cP} & \cP \ar[d]^{\id: \cP \xt G \to \cP \xt G}  
\ar[l]_-p \ar[d] \ar[r]^-f & Z \ar[dd]^{\id: Z \xt G \to Z \xt G} \\ 
&  [\cP \xt G/G] \ar[d]^{(f \xt \id, \alpha_\cP)} \\ 
[\cP/G] & \ds [((Z \xt G) \xt_{\alpha, Z, f} \cP)/G] \ar[r]\ar[l] & [Z \xt G/G] } \]
is $2$-commutative. Applying Remark~\ref{P/G} multiple times, we deduce that
\eqref{eqn:Z-Z/G:fiber} is isomorphic to the desired Cartesian square of the claim. 
Moreover, we see that the morphism $f^*\tau_0 \to p^*\tau$ is defined by the
action and first projection morphisms. 
\end{proof}

\begin{rem}\label{BGxZ}
Let $Z$ be a $k$-scheme and give it the trivial $G$-action. Lemma~\ref{Z/G-atlas} 
implies that $\cdot \to BG$ is $G$-invariant, so by base change $Z \to BG \xt Z$
is also $G$-invariant. For any morphism $(\cP,u) : S\to BG \xt Z$, the fibered product
$Z \xt_{BG \xt Z} S \simeq \cdot \xt_{BG} S$ is isomorphic to $\cP$ by Corollary~\ref{Z-Z/G},
and the morphism $\cP \to Z$ is $\cP \to S$ composed with $u$. Lemma~\ref{Z/G=X} then gives
an isomorphism $BG \xt Z \simeq [Z/G]$. This isomorphism implies that for a $G$-bundle
$\cP$ on $S$, any $G$-invariant morphism $\cP \to Z$ factors through a unique morphism $S \to Z$.
\end{rem}

\subsection{Change of group} \label{section:cog}
Let $H \into G$ be a closed subgroup of $G$. We want to show the following relation
between $BH$ and $BG$. 

\begin{lem} \label{BH-BG} The morphism $BH \to BG$ sending $\cP \mapsto \twist \cP G$ is schematic,
finitely presented, and quasi-projective.
\end{lem} 

\noindent Here we are twisting $G$ by the left $H$-action. As we shall see, 
$\cP \mapsto \twist \cP G$ gives a well-defined morphism of stacks $BH \to BG$.
\medskip

By \cite[III, \S 3, Th\'eor\`eme 5.4]{De-Ga}, 
the fppf sheaf $H\bs G$ of left cosets is representable by a quasi-projective $k$-scheme with
a $G$-equivariant, ample invertible sheaf. 
The morphism $H \xt G \to G \xt_{H\bs G} G$ induced by multiplication and projection is 
an isomorphism 
by \cite[III, \S 1, 2.4]{De-Ga}. Therefore the projection $\pi :G \to H\bs G$ is a left $H$-bundle.
Note that there is an obvious analogue of our previous discussion of quotient stacks
for schemes with left actions instead of right actions. We use $[H \bs G]$ to denote
the stack quotient by the left action.  
Then Remark~\ref{P/G} implies that $\id : G \to G$ induces an isomorphism $H\bs G \to [H\bs G]$. 

\begin{lem} \label{HG/G=BH} 
The left $H$-bundle $G \to H\bs G$ defines a right $G$-bundle $H\bs G \to [H\bs \cdot]$. 
\end{lem}
\begin{proof}
Let $\mu : G \xt G \to G$ denote the group multiplication.
We have an $H$-equivariant morphism $(\pr_1, \mu) : G \xt G \to G \xt G$ where $H$ acts 
on the first coordinate on the left hand side, and $H$ acts diagonally on the right hand side.
We have $2$-commutative squares 
\[ \xymatrix{ H \bs G \ar[d]_{\id : G \to G} & H\bs G \xt G \ar[l]_{\pr_1} \ar[r]^{\wbar\mu} 
\ar[d]|{(\pr_1,\mu) : G \xt G \to G \xt G} & H \bs G \ar[d]^{\id : G \to G} \\ 
[H\bs G] & [H \bs (G \xt G)] \ar[l]_{\pr_1} \ar[r]^{\pr_2} & [H \bs G] } \]
where the $2$-morphisms are just $\id_{G \xt G}$. Thus we have a $2$-commutative square
\[ \xymatrix{ H \bs G \xt G \ar[r]^{\wbar \mu} \ar[d]_{\pr_1} & H \bs G \ar[d] \\ 
H\bs G \ar[r] & [H \bs \cdot] } \]
with $2$-morphism $\id_{G\xt G}$, which gives $H \bs G \to [H \bs \cdot]$ the structure of
a right $G$-invariant morphism. Take a left $H$-bundle $\cP \in [H\bs \cdot](S)$ and an fppf 
covering $(S_i \to S)$ trivializing $\cP$. From the proof of
Lemma~\ref{cos-representable}, we find that
\[ \xymatrix{ S_i \xt G \ar[d] \ar[r] & [H \bs G] \ar[d] \\ 
S_i \ar[r]^{H \xt S_i} & [H \bs \cdot] } \]
is Cartesian, the top arrow corresponds to $H \xt S_i \xt G \to G : (h,a,g) \mapsto hg$,
and the associated $2$-morphism is $\id_{H \xt S_i \xt G}$. Observe that
the morphism $S_i \xt G \to [H \bs G] \simeq H \bs G$ equals $\pi \circ \pr_2$, 
which is $G$-equivariant. 
Therefore $S_i \xt G \to H\bs G \xt_{[H\bs \cdot], H \xt S_i} S_i \simeq (H \bs G \xt_{[H\bs \cdot],\cP} S) \xt_S S_i$ is a $G$-equivariant isomorphism.
Thus $H \bs G \xt_{[H\bs \cdot],\cP} S$ is a right $G$-bundle, 
which proves that $H\bs G \to [H\bs \cdot]$ is also a right $G$-bundle. 
\end{proof}

Now Lemmas \ref{Z/G=X} and \ref{HG/G=BH} give an isomorphism $[H\bs \cdot] \simeq [(H\bs G)/G]$
sending a left $H$-bundle $\cP$ over $S$ to $H\bs G \xt_{[H \bs \cdot],\cP} S \to H \bs G$.
Left $H$-bundles are equivalent to
right $H$-bundles using the inverse action, so $BH \simeq [H\bs \cdot]$. 
By keeping track of left and right actions and applying Corollary~\ref{Z/G-BG}, we 
observe that the morphism $BH \simeq [H \bs \cdot] \simeq [(H \bs G)/G] \to BG$
sends a right $H$-bundle $\cP$ to the associated fiber bundle $\twist \cP G$, 
where $G$ is twisted by the left $H$-action. The right $G$-action on $\twist \cP G$ is
induced by multiplication on the right by $G$.

\begin{proof}[Proof of Lemma~\ref{BH-BG}] 
Applying
Corollary~\ref{Z/G-BG} to the morphism $[(H\bs G)/G] \to BG$, we get a $2$-Cartesian square
\[ \xymatrix{ \twist \cE{(H \bs G)} \ar[r] \ar[d] & BH \ar[d] \\ S \ar[r]^\cE & BG } \]
Since $H \bs G$ is quasi-projective with a $G$-equivariant ample 
invertible sheaf over $\spec k$, we have that $\twist \cE {(H\bs G)}$ is representable by a 
scheme quasi-projective and of finite presentation over $S$ by Lemma~\ref{cos-representable}. 
\end{proof}

\subsection{Proving $[Z/G]$ is algebraic} We now prove Theorem~\ref{Z/G-alg}.

\begin{lem}\label{BG^2} 
For algebraic groups $G, G'$, the morphism
 $BG \xt BG' \to B(G\xt G')$ sending a $G$-bundle $\cP$ and a $G'$-bundle $\cP'$ over $S$ to 
$\cP \xt_S \cP'$ is an isomorphism.
\end{lem}
\begin{proof}
Take $(\cP,\cP') \in (BG \xt BG')(S)$ and choose an fppf covering $(S_i \to S)$ 
on which both $\cP$ and $\cP'$ are trivial. Then we have descent data 
\[ (S_i \xt G, g_{ij}),\; (S_i \xt G', g'_{ij}) \] 
for $g_{ij} \in G(S_i \xt_S S_j), g'_{ij} \in G'(S_i \xt_S S_j)$,
corresponding to $\cP,\cP'$, respectively. Then 
\[ (S_i \xt G \xt G' , (g_{ij},g'_{ij}) ) \]
gives a descent datum for $\cP \xt_S \cP'$. Conversely, any such descent datum 
with $(g_{ij},g'_{ij})$ satisfying the cocycle condition implies that $g_{ij}$ and $g'_{ij}$
satisfy the cocycle condition separately. Since $BG \xt BG'$ and $B(G \xt G')$ are
both stacks, we deduce that $(\cP, \cP') \mapsto \cP \xt_S \cP'$ gives an isomorphism. 
From the descent data, we see that an inverse to this morphism is the morphism 
sending $\cE \mapsto (\twist \cE G, \twist \cE G')$, where $G \xt G'$ acts by the
projections on $G,G'$.
\end{proof}

\begin{lem}\label{BG-alg}
The $k$-stack $BG$ is an algebraic stack with a schematic, affine diagonal. 
Specifically, for right $G$-bundles $\cP, \cP'$ over $S$, there is an isomorphism 
\[ \Isom_{BG(S)}(\cP,\cP') \simeq \twist{\cP \xt_S \cP'}{(G)} \] as sheaves of sets on $\Sch_{/S}$, 
where $G\xt G$ acts on $G$ from the right by $g.(g_1,g_2) = g_1^{-1} g g_2$. 
\end{lem}
\begin{proof}
For a scheme $S$ and $\cP, \cP' \in BG(S)$, we have a Cartesian square
\[ \xymatrix{ \Isom_{BG(S)}(\cP,\cP') \ar[r] \ar[d] & BG \ar[d] \\ S \ar[r]^{\cP,\cP'} & 
BG \xt BG } \]
Observe that there is an isomorphism $G \bs (G \xt G) \simeq G$ by sending $(g_1,g_2) \mapsto
g_1^{-1} g_2$. The induced right action of $G \xt G$ on $G$ is then $g.(g_1,g_2) = g_1^{-1}g g_2$
as claimed. By the proofs of Lemmas~\ref{BH-BG} and \ref{BG^2}, we get an isomorphism 
$\Isom_{BG(S)}(\cP,\cP') \simeq \twist{\cP \xt_S \cP'}{(G)}$ over $S$. 
Since $G$ is affine over $k$, the associated fiber bundle
$\twist{\cP \xt_S \cP'}{(G)}$ is representable by a scheme affine over $S$.
Hence the diagonal $BG \to BG \xt BG$ is schematic and affine.
As a special case of Lemma~\ref{Z-Z/G}, the morphism $\cdot\to BG$ is schematic, affine, and fppf. 
By Artin's Theorem \cite[Th\'eor\`eme 10.1]{LMB}, $BG$ is an algebraic stack 
(if $G$ is smooth over $k$, the last step is unnecessary).
\end{proof}

\begin{proof}[Proof of Theorem~\ref{Z/G-alg}]
Fix a scheme $S$ and $(\cP \to Z) \in [Z/G](S)$. Then by Remark~\ref{P/G} and 
Lemma~\ref{cos-representable}, the morphism $S \simeq [\cP/G] \to [Z/G]$ is representable. 
Therefore the diagonal $\Delta_{[Z/G]}$ is representable. 
By Lemma~\ref{BG-alg}, there exists a scheme $U$ and a smooth surjective morphism $U \to BG$. 
The change of space morphism $f:[Z/G] \to BG$ is representable by Lemma~\ref{cos-representable}, 
so $U \xt_{BG} [Z/G]$ is representable by an algebraic space. Therefore there exists a scheme
$V$ with an \'etale surjective morphism $V \to U \xt_{BG} [Z/G]$. The composition $V \to [Z/G]$
thus gives a presentation. This shows that $[Z/G]$ is an algebraic stack. 

The diagonal $\Delta_{[Z/G]}$ is the composition of the morphisms 
\[ [Z/G] \overset{\Delta_f}\lto [Z/G] \xt_{BG} [Z/G] \to [Z/G] \xt [Z/G].\] 
We have a $2$-Cartesian square 
\[ \xymatrix{ \ds [Z/G] \xt_{BG} [Z/G] \ar[r] \ar[d] & [Z/G] \xt [Z/G] \ar[d] \\ 
BG \ar[r] & BG \xt BG } \]
and we know the diagonal $\Delta_{BG}$ is schematic and affine by Lemma~\ref{BG-alg}. 
By base change, this implies $[Z/G] \xt_{BG} [Z/G] \to [Z/G] \xt [Z/G]$ is also schematic 
and affine. Since $f$ is representable, 
\cite[Lemma \href{http://math.columbia.edu/algebraic_geometry/stacks-git/locate.php?tag=04YQ}{04YQ}]{stacks-project} implies that $\Delta_f$ is schematic and separated. 
If $Z$ is quasi-separated, then by descent $\twist \cP Z \to S$ is quasi-separated for any 
$\cP \in BG(S)$. Thus Corollary~\ref{Z/G-BG} implies that $f$ is quasi-separated,
and \cite[Lemma \href{http://math.columbia.edu/algebraic_geometry/stacks-git/locate.php?tag=04YT}{04YT}]{stacks-project} shows that $\Delta_f$ is quasi-compact. 
If $Z$ is separated, we deduce by the same reasoning that $f$ is separated.
Hence by \cite[Lemma \href{http://math.columbia.edu/algebraic_geometry/stacks-git/locate.php?tag=04YS}{04YS}]{stacks-project}, the relative diagonal $\Delta_f$ is a closed immersion.
Composition of the two morphisms $\Delta_f$ and $[Z/G] \xt_{BG} [Z/G] \to [Z/G] \xt [Z/G]$ 
now gives the claimed properties of $\Delta_{[Z/G]}$. 
\end{proof}

\section{Hom stacks} \label{section:hom-stacks}
For a base scheme $S$, let $X$ be an $S$-scheme and $\cY : (\Sch_{/S})^\op \to \Gpd$
a pseudo-functor.
Then we define the Hom $2$-functor $\cHom_S(X,\cY)$ by
\[ \cHom_S(X,\cY)(T) = \Hom_T(X_T, \cY_T) = \Hom_S(X_T, \cY), \]
sending an $S$-scheme $T$ to the $2$-category of groupoids.
By the $2$-Yoneda lemma, we have a natural equivalence of categories 
$\Hom_S(X_T,\cY) \simeq \cY(X_T)$. From this we deduce that if $\cY$ is an fpqc $S$-stack,
then $\cHom_S(X,\cY)$ is as well.
\smallskip

The main example of a Hom stack we will be concerned with is $\Bun_G$:

\begin{defn} \label{defn:BunG} 
Let $S$ be a $k$-scheme. Then for an $S$-scheme $X$, we define the moduli stack of 
$G$-bundles on $X \to S$ by $\Bun_G = \cHom_S(X,BG \xt S)$. 
\end{defn}

As we remarked earlier, $\Bun_G$ is an fpqc stack since $BG$ is. 
Explicitly, for an $S$-scheme $T$, $\Bun_G(T)$ is the groupoid of right $G$-bundles on $X_T$.
\smallskip

In this section, we prove some properties on morphisms between Hom stacks. 
As a corollary, we show that the diagonal of $\Bun_G$ is schematic and affine 
under certain conditions. These properties reduce
to understanding the sheaf of sections associated to a morphism of schemes, which we now
introduce. 

\subsection{Scheme of sections} \label{section:sect}
Let $S$ be a base scheme and 
$X \to S$ a morphism of schemes. Given another morphism of schemes $Y \to X$, 
we define the presheaf 
of sets $\Sect_S(X,Y)$ on $\Sch_{/S}$ to send 
\[ (T \to S) \mapsto \Hom_{X_T}(X_T, Y_T) = \Hom_X(X_T, Y). \]
Since schemes represent fpqc sheaves, 
the presheaf $\Sect_S(X,Y)$ is an fpqc sheaf of sets. 
Here is the main result of this subsection:

\begin{thm}\label{sect-general} 
Let $X \to S$ be a flat, finitely presented, projective morphism, and let $Y \to X$ be a
finitely presented, quasi-projective morphism. Then $\Sect_S(X,Y)$ is representable by a 
disjoint union of schemes which are finitely presented and
locally quasi-projective over $S$. 
\end{thm}

\begin{eg} Observe that if $X$ and $Z$ are schemes over $S$, then taking the
first projection $X \xt_S Z \to X$ gives an equality $\Sect_S(X, X \xt_S Z) = 
\cHom_S(X,Z)$. \end{eg}

We first consider when $Y \to X$ 
is affine, which will be of independent interest because we get a stronger result. 
We then prove the theorem by considering the cases where $Y \to X$ is an open immersion 
and a projective morphism separately. 

\begin{lem}\label{sect-vb}
Let $p:X \to S$ be a flat, finitely presented, proper morphism and $\cE$ a locally free 
$\cO_X$-module of finite rank. Then $\Sect_S(X,\rspec_X \sym_{\cO_X} \cE)$ is 
representable by a scheme affine and finitely presented over $S$. 
\end{lem}
\begin{proof}
By \cite[0, Corollaire 4.5.5]{EGA1-2ed}, we can take an open covering of $S$ and reduce to 
considering the case where $S$ is affine. Since any ring is an inductive limit of 
finitely presented $\bZ$-algebras, by \cite[\S 8]{EGA4c} there exist an affine
scheme $S_1$ of finite type over $\spec \bZ$ with a morphism $S \to S_1$, a flat proper
morphism $X_1 \to S_1$, and a locally free $\cO_{X_1}$-module $\cE_1$ such that 
$X_1, \cE_1$ base change to $X, \cE$ under $S \to S_1$. Therefore 
\[ \Sect_S(X,\rspec_X \sym_{\cO_X} \cE) \simeq \Sect_{S_1}(X_1, \rspec_{X_1} \sym_{\cO_{X_1}} 
\cE_1) \xt_{S_1} S \]
over $S$. Replacing $S$ by $S_1$, we can assume that $S$ is affine and noetherian.

Let $\cF = \cE^\vee$, which is a locally free $\cO_X$-module. 
From \cite[Th\'eor\`eme 7.7.6, Remarques 7.7.9]{EGA3b}, there exists a coherent $\cO_S$-module 
$\cQ$ equipped with natural isomorphisms 
\begin{equation}\label{eqn:sect-vb}
 \Hom_{\cO_T}(\cQ_T,\cO_T) \simeq \Gamma(T, p_{T*}\cF_T) = \Gamma(X_T,\cF_T)
\end{equation}
for any $S$-scheme $T$. 
A section $X_T \to (\rspec_X \sym \cE) \xt_S T$ over $X_T$ is equivalent to 
a morphism of $\cO_{X_T}$-modules $\cE_T \to \cO_{X_T}$. The latter is by definition 
a global section of $\Gamma(X_T, \cE_T^\vee)$. 
Since $\cE$ is locally free, we have a canonical isomorphism $\cF_T \simeq \cE_T^\vee$. 
The isomorphism of \eqref{eqn:sect-vb} 
therefore shows that an element of $\Sect_S(X,\rspec_X \sym \cE)(T)$ is naturally
isomorphic to a morphism in $\Hom_S(T, \rspec_S \sym_{\cO_S}\cQ) \simeq \Hom_{\cO_T}(\cQ_T,\cO_T)$.
We conclude that 
\[ \Sect_S(X,\rspec_X \sym_{\cO_X} \cE) \simeq \rspec_S \sym_{\cO_S} \cQ \]
as sheaves over $S$. 
\end{proof}


\begin{lem}\label{sect-affine}
Let $p: X \to S$ be a flat, finitely presented, projective morphism. 
If $Y \to X$ is affine and finitely presented, then $\Sect_S(X,Y)$ is representable by a 
scheme affine and finitely presented over $S$. 
\end{lem} 
\begin{proof}
We make the same reductions as in the proof of Lemma~\ref{sect-vb} to assume $S$ is affine
and noetherian, and $X$ is a closed subscheme of $\bP^r_S$ for some integer $r$. 
Since $Y$ is affine over $X$, it has the form $Y = \rspec_X \cA$ for a quasi-coherent
$\cO_X$-algebra $\cA$. By finite presentation of $Y$ over $X$
and quasi-compactness of $X$, there are finitely many local sections of $\cA$
that generate $\cA$ as an $\cO_X$-algebra. 
By extending coherent sheaves \cite[Corollaire 9.4.3]{EGA1}, there exists 
a coherent $\cO_X$-module $\cF \subset \cA$ that contains all these sections. 
From \cite[II, Corollary 5.18]{Hart}, there exists a locally free sheaf $\cE_1$ on $X$ 
surjecting onto $\cF$. This induces a surjection of $\cO_X$-algebras 
$\sym_{\cO_X} \cE_1 \onto \cA$. Let $\cI \subset \sym \cE_1$ be the ideal sheaf. 
Apply the same argument again to get a locally free sheaf $\cE_2$ on $X$ and a morphism 
$\cE_2 \to \cI$ whose image generates $\cI$ as a $\sym_{\cO_X} \cE_1$-module. Therefore
we have a Cartesian square 
\[ \xymatrix{ Y = \rspec_X \cA \ar[r] \ar[d] & X \ar[d] \\ \rspec_X \sym \cE_1 \ar[r] & 
\rspec_X \sym \cE_2 } \]
over $X$, where $X \to \rspec_X \sym \cE_2$ is the zero section. By the universal
property of the fibered product, we have a canonical isomorphism
\[ \Sect_S(X,Y) \simeq \Sect_S(X,X) \xt_{\Sect_S(X,\rspec_X \sym \cE_2)} \Sect_S(X,\rspec_X \sym \cE_1). \]
Note that $S \simeq \Sect_S(X,X)$. 
By Lemma~\ref{sect-vb}, all three sheaves in the fibered product are representable by schemes
affine and finitely presented over $S$. Therefore $\Sect_S(X,Y)$ is representable by a scheme.
The morphism $S \simeq \Sect_S(X,X) \to 
\Sect_S(X, \rspec_X \sym \cE_2)$ sends $T \to S$ to $0 \in \Gamma(X_T, (\cE_2)_T^\vee)$, using 
the correspondence in Lemma~\ref{sect-vb}. Therefore 
\[ S \to \Sect_S(X,\rspec_X \sym \cE_2) \simeq \rspec_S \sym \cQ_2 \]
is the zero section, where $\cQ_2$ is as in Lemma~\ref{sect-vb}. In particular, it is a closed
immersion. By base change, $\Sect_S(X,Y) \to \Sect_S(X, \rspec_X \sym \cE_1)$ is also a closed
immersion. Therefore $\Sect_S(X,Y)$ is affine and finitely presented over $S$. 
\end{proof}

\begin{rem} \label{rem:sect-closed}
If $Y \to X$ is a finitely presented closed immersion, 
then note that we can take $\cE_1 = 0$ in the proof of Lemma~\ref{sect-affine}. 
It follows that $\Sect_S(X,Y) \to S$ is a closed immersion. 
\end{rem} 

\begin{lem} \label{sect-open}
Let $p: X \to S$ be a proper morphism. For any morphism $Y \to X$ and $U \into Y$ an 
open immersion, $\Sect_S(X,U) \to \Sect_S(X,Y)$ is schematic and open. 
\end{lem}
\begin{proof}
For an $S$-scheme $T$, 
suppose we have a morphism $f: X_T \to Y$ over $S$ corresponding to a morphism 
of sheaves $T \to \Sect_S(X,Y)$ by the Yoneda lemma. Then the fibered product
\[ T \xt_{\Sect_S(X,Y)} \Sect_S(X,U) \] is a sheaf on $\Sch_{/T}$ sending $T' \to T$ to 
a singleton set if $X_{T'} \to X_T \to Y$ factors through $U$ and to the empty set otherwise. 
We claim that this sheaf is representable by the open subscheme \[ T - p_T(X_T - f^{-1}(U))\]
of $T$, where $p_T(X_T - f^{-1}(U))$ is a closed subset of $T$ by properness of $p_T$. 
Suppose that the image of $T' \to T$ contains a point of $p_T(X_T - f^{-1}(U))$. 
In other words, there are $t' \in T'$ and $x \in X_T$ mapping to the same point of $T$,
and $f(x) \notin U$. Then the compositum field of $\kappa(t'), \kappa(x)$ over $\kappa(p_T(x))$
gives a point of $X_{T'}$ mapping to $t',x$. This shows that $X_{T'} \to Y$ does not 
factor through $U$. For the converse, suppose that there is a point $x' \in X_{T'}$ with 
image outside of $U$. Then the image $x \in X_T$ is not in $f^{-1}(U)$, so $p_{T'}(x')$ maps 
to $p_T(x) \in p_T(X_T- f^{-1}(U))$. We conclude that $T' \to T$ factors through 
$T - p_T(X_T - f^{-1}(U))$ if and only if $X_{T'} \to Y$ factors through $U$. 
Therefore $T \xt_{\Sect_S(X,Y)} \Sect_S(X,U)$ is representable by this open subscheme
of $T$. 
\end{proof}

\begin{lem} \label{locus-iso} Let $S$ be a noetherian scheme, and let $p : X \to S$
and $Z \to S$ be flat proper morphisms. Suppose there is a morphism $\pi : Z \to X$ over $S$.
Then there exists an open subscheme $S_1 \subset S$ with the following universal
property. For any locally noetherian $S$-scheme $T$, the base change 
$\pi_T : Z_T \to X_T$ is an isomorphism if and only if $T \to S$ factors through $S_1$. 
\end{lem}
\begin{proof}
Since $X$ and $Z$ are both proper over $S$, we deduce that $\pi$ is a proper morphism. 
By \cite[Theorem 5.22(a)]{FGA-explained}, we can assume that $\pi$ is flat. 
From Chevalley's upper semi-continuity theorem for dimension of fibers
 \cite[Corollaire 13.1.5]{EGA4c}, the set of $x \in X$ such that $\pi^{-1}(x)$ is finite
form an open subset $U_1 \subset X$. The restriction $\pi : \pi^{-1}(U_1) \to U_1$ is 
a flat, finite morphism by \cite[Proposition 4.4.2]{EGA3a}. Therefore $\pi_*\cO_{\pi^{-1}(U_1)}$
is a locally free $\cO_{U_1}$-module. Let $U$ be the set of $x \in U_1$ where 
$\cO_{U_1} \ot \kappa(x) \to \pi_*\cO_{\pi^{-1}(U_1)} \ot \kappa(x)$ is an isomorphism. 
Since $\cO_{U_1} \to \pi_*\cO_{\pi^{-1}(U_1)}$ is a morphism between 
locally free $\cO_{U_1}$-modules, $U$ is the maximal open subscheme of $X$ 
on which this morphism is an isomorphism, which is equivalent to being the maximal open such that
$\pi^{-1}(U) \to U$ is an isomorphism. Since the property of being an isomorphism is
fpqc local on the base, the same argument as in the proof of 
Lemma~\ref{sect-open} shows that $S_1 = S - p(X-U)$ is the open subscheme with the desired
universal property. 
\end{proof}

Suppose we have a proper morphism $X \to S$ and a separated morphism $Y \to X$. 
Then for a section $f \in \Sect_S(X,Y)(T)$, the graph of $f$ over $X_T$ is
a closed immersion $X_T \to X_T \xt_{X_T} Y_T$. The isomorphism $X_T \xt_{X_T} Y_T \to Y_T$
implies $f$ is a closed immersion. Therefore $f : X_T \into Y_T$ represents an element of
$\Hilb_{Y/S}(T)$. Therefore we have defined an injection of sheaves 
\begin{equation}\label{eqn:sect-hilb} \Sect_S(X,Y) \to \Hilb_{Y/S}. \end{equation}

\begin{lem}\label{sect-proper-hilb} 
Let $p: X\to S$ and $Y \to S$ be finitely presented, proper morphisms, and suppose that $p$
is flat. Then \eqref{eqn:sect-hilb} is an open immersion.
\end{lem}
\begin{proof} The assertion is Zariski local on the base, so 
we can use \cite[\S8]{EGA4c} as in the proof of Lemma~\ref{sect-vb} to 
assume $S$ is affine and noetherian. Let $T \to \Hilb_{Y/S}$ 
represent $Z\subset Y_T$, a closed subscheme flat over an affine $S$-scheme $T$. 
We may assume $T$ is noetherian by \cite[\S 8]{EGA4c}.
Applying Lemma~\ref{locus-iso} to the composition $Z \to X_T$, we deduce that
there exists an open subscheme $U \subset T$ such that for any locally noetherian scheme $T'
\to T$, the base change $Z_{T'} \to X_{T'}$ is an isomorphism if and only if $T' \to T$
factors through $U$. Observe that an isomorphism $Z_{T'} \to X_{T'}$ corresponds 
uniquely to a section $X_{T'} \simeq Z_{T'} \into Y_{T'}$. The locally noetherian
hypothesis on $T'$ can be removed as usual. Therefore 
$T \xt_{\Hilb_{Y/S}} \Sect_S(X,Y)$ is represented by $U$. 
\end{proof}

\begin{proof}[Proof of Theorem~\ref{sect-general}] 
Let $\cL$ be an invertible $\cO_Y$-module ample relative to $\pi : Y \to X$ 
and $\cK$ an invertible $\cO_X$-module ample relative to $p: X \to S$. 
Choose an integer $\chi_{n,m}$ for every pair of integers $n,m$. 
Define the subfunctor
\[ \Hilb_{Y/S}^{(\chi_{n,m})} \subset \Hilb_{Y/S} \]
to have $T$-points the closed subschemes $Z \in \Hilb_{Y/S}(T)$ such that
for all $t \in T$, the Euler characteristic 
$\chi((\cO_Z \ot \cL^{\otimes n} \ot \pi^*(\cK^{\otimes m}))_t) = \chi_{n,m}$. 
We claim that the $\Hilb_{Y/S}^{(\chi_{n,m})}$ form a disjoint open cover of 
$\Hilb_{Y/S}$. This can be checked Zariski locally on $S$, so we can assume $S$ is
noetherian using \cite[\S 8]{EGA4c}. Then $\Hilb_{Y/S}$ is representable by a locally
noetherian scheme \cite[Corollary 2.7]{AK80}. 
Since the Euler characteristic is locally constant \cite[Th\'eor\`eme 7.9.4]{EGA3b}
and the connected components of a locally noetherian scheme are open 
\cite[Corollaire 6.1.9]{EGA1}, we deduce the claim. 

Now it suffices to show that for any choice of $(\chi_{n,m})$, the functor
\[ \Sect_S(X,Y) \cap \Hilb_{Y/S}^{(\chi_{n,m})} \] 
is representable by a scheme 
finitely presented and locally quasi-projective over $S$. The assertion is Zariski local on $S$
\cite[0, Corollaire 4.5.5]{EGA1-2ed}, so we can assume that $S$ is affine and noetherian.  
Now by \cite[Propositions 4.4.6, 4.6.12, 4.6.13]{EGA2}, there exists a scheme $\wbar Y$
projective over $X$, an open immersion $Y \into \wbar Y$, an invertible module
$\wbar \cL$ very ample relative to $\wbar Y \to S$, and positive integers $a,b$
such that \[ \cL^{\otimes a} \ot \pi^*( \cK^{\otimes b}) \simeq \wbar \cL |_Y.\]
Let $\Phi \in \bQ[\lambda]$ be a polynomial such that $\Phi(n) = \chi_{na,nb}$ for all
integers $n$ (if no such polynomial exists, then $\Hilb_{Y/S}^{(\chi_{n,m})}$ is empty). 
Lemmas~\ref{sect-open} and \ref{sect-proper-hilb} and \cite[Theorem 2.6, Step IV]{AK80} 
imply that $\Sect_S(X,Y) \cap \Hilb_{Y/S}^{(\chi_{n,m})}$ is
an open subfunctor of $\Hilb_{\wbar Y/S}^{\Phi,\wbar \cL}$. 
The claim now follows from \cite[Corollary 2.8]{AK80}
as an open immersion to a noetherian scheme is finitely presented.
\end{proof}

\subsection{Morphisms between Hom stacks}
The goal of this subsection is to use our results on the scheme of sections to deduce
that the diagonal of $\Bun_G$ is schematic when $X$ is flat, finitely presented, and projective
over $S$. 

\begin{lem} \label{morphism-hom}
Suppose that $X \to S$ is flat, finitely presented, and projective. 
Let $F:\cY_1 \to \cY_2$ be a schematic 
morphism between pseudo-functors. If $F$ is quasi-projective (resp.~affine) and
of finite presentation, then the corresponding morphism 
\[ \cHom_S(X,\cY_1) \to \cHom_S(X,\cY_2) \] 
is schematic and locally of finite presentation 
(resp.~affine and of finite presentation). 
\end{lem}

Before proving the lemma, we mention the corollaries of interest. 

\begin{cor} \label{BunG-diagonal} Suppose that $X\to S$ is flat, finitely presented, and 
projective. Then the diagonal of $\Bun_G$ is schematic, affine, and finitely presented.
\end{cor}
\begin{proof}
By Lemma~\ref{BG^2}, we deduce that the canonical morphism 
$\Bun_{G \xt G} \to \Bun_G \xt \Bun_G$ is an isomorphism.
Applying Lemmas~\ref{BG-alg} and \ref{morphism-hom} to $BG \to 
B(G \xt G)$, we deduce that $\Bun_G \to \Bun_{G\xt G}$ is schematic, affine,
and finitely presented, which proves the claim. 
\end{proof}

\begin{rem} 
In line with Remark~\ref{rem:general-group-scheme}, we note that Corollary~\ref{BunG-diagonal}
holds if $G$ is a group scheme affine and of finite presentation over a base $S$. 
\end{rem}

\begin{cor} \label{BunG-cog}
Let $H \into G$ be a closed subgroup of $G$. If $X \to S$ is flat, finitely presented,
and projective, then the corresponding morphism 
$\Bun_H \to \Bun_G$ is schematic and locally of finite presentation. 
\end{cor}
\begin{proof}
This follows from Lemmas~\ref{BH-BG} and \ref{morphism-hom}. 
\end{proof}

\begin{proof}[Proof of Lemma~\ref{morphism-hom}]
Let $\tau_2 : X_T \to \cY_2$ represent a morphism from $T \to \cHom_S(X,\cY_2)$. 
Since $F$ is schematic, the $2$-fibered product $\cY_1 \xt_{\cY_2,\tau_2} X_T$ is
representable by a scheme $Y_T$. Let 
\[ \xymatrix{ Y_T \ar[r]^{\tau_1} \ar[d]_\pi & \cY_1 \ar[d]^F \\ X_T \ar[r]^{\tau_2} & \cY_2 } \]
be the $2$-Cartesian square, with a $2$-morphism $\gamma : F(\tau_1) \simeq \pi^*\tau_2$. 
Now for a $T$-scheme $T'$, suppose that $\tau'_1 : X_{T'} \to \cY_1$ is a $1$-morphism such that 
the square 
\[ \xymatrix{ X_{T'} \ar[r]^{\tau'_1} \ar[d]_{\pr_1} & \cY_1 \ar[d]^F \\ X_T \ar[r]^{\tau_2} & 
\cY_2 } \]
is $2$-commutative via a $2$-morphism $\gamma' : F(\tau'_1) \simeq \pr_1^*\tau_2$.
By the definition of $2$-fibered products and the assumptions on $Y_T$, there exists a unique
morphism of schemes $f : X_{T'} \to Y_T$ over $X_T$ and a unique $2$-morphism 
$\phi : f^* \tau_1 \simeq \tau'_1$ such that 
\[ \xymatrix{ F(f^*\tau_1) \ar[r]^\sim \ar[d]_{F(\phi)}  & f^*F(\tau_1) \ar[r]^{f^*\gamma} & 
f^*\pi^*\tau_2 \ar[d]^\sim \\ 
F(\tau'_1) \ar[rr]^{\gamma'} && \pr_1^* \tau_2 } \]
commutes. On the other hand, we have that
\[ \Bigl(\cHom_S(X,\cY_1) \xt_{\cHom_S(X,\cY_2)} T \Bigr)(T') = 
\bigl\{ ( T' \to T,\, \tau'_1 : X_{T'} \to \cY_1,\,
\gamma' : F(\tau'_1) \simeq \pr_1^*\tau_2 ) \bigr\} \]
by the $2$-Yoneda lemma. Thus for a fixed $T$-scheme $T'$ and a pair $(\tau'_1, \gamma')$,  
there is a unique $f \in \Hom_{X_T}(X_{T'},Y_T) = \Sect_T(X_T,Y_T)(T')$ such that 
\[ \bigl(f^*\tau_1, f^*\gamma : F(f^*\tau_1) \simeq \pr_1^*\tau_2 \bigr) 
\in \Bigl( \cHom_S(X,\cY_1) \xt_{\cHom_S(X,\cY_2)} T \Bigr)(T') \]
and there is a unique $2$-morphism $(f^*\tau_1, f^*\gamma) \simeq (\tau'_1, \gamma')$
induced by $\phi$. Therefore we have a Cartesian square
\[ \xymatrix{ \Sect_T(X_T, Y_T) \ar[r] \ar[d] & \cHom_S(X,\cY_1) \ar[d] \\ T \ar[r] & 
\cHom_S(X,\cY_2) } \]
We have that $Y_T \to X_T$ is finitely presented and quasi-projective (resp.~affine).
By Lemma~\ref{sect-affine} and
Theorem~\ref{sect-general}, we deduce that $\Sect_T(X_T,Y_T) \to T$ 
is schematic and locally of finite presentation (resp.~affine and of finite presentation). 
\end{proof}

\section{Presentation of \texorpdfstring{$\Bun_G$}{BunG}} \label{section:Bun_G-presentation}

Recall the definition of $\Bun_G$ from Definition~\ref{defn:BunG}. 
In this section, we prove Theorem~\ref{BunG-alg}. To do this, we embed $G$ in $\GL_r$ 
and reduce to giving a presentation of the moduli stack of locally free modules of rank $r$.

\subsection{Proving \texorpdfstring{$\Bun_G$}{BunG} is algebraic} 
The general technique we use to prove results on $G$-bundles is to reduce to working with 
$\GL_r$-bundles, which are particularly nice because
$\GL_r$-bundles are in fact equivalent to locally free modules, or  
vector bundles, of rank $r$. 

\begin{lem} \label{GL_r-bundle} 
The morphism from the $k$-stack
\[ B_r: T \mapsto \{ \mbox{locally free $\cO_T$-modules of rank $r$} \} + \{\mbox{isomorphisms of
$\cO_T$-modules} \} \]
to $B\GL_r$ sending $\cE \mapsto \Isom_T(\cO_T^r, \cE)$ is an isomorphism.
\end{lem}
\begin{proof}
First observe that $B_r$ is an fpqc stack because $\QCoh$ is an fpqc stack 
\cite[Theorem 4.2.3]{FGA-explained} and local freeness of rank $r$ persists under
fpqc morphisms \cite[Proposition 2.5.2]{EGA4b}. 
There is a canonical simply transitive right action of $\GL_r$ on $\Isom_T(\cO_T^r, \cE)$
by composition. Take a Zariski covering $(T_i \subset T)$ trivializing $\cE$. 
Then $\cE$ has a descent datum $(\cO_{T_i}^r, g_{ij})$ for $g_{ij} \in \GL_r(T_i \cap T_j)$. 
Since $\Isom_{T_i}(\cO_{T_i}^r, \cO_{T_i}^r) \simeq T_i \xt \GL_r$, we have
a descent datum $(T_i \xt \GL_r, g_{ij})$ corresponding to $\Isom_T(\cO_T^r, \cE)$.
Conversely, for any $\cP \in B\GL_r(T)$, there exists a descent datum 
$(T_i \xt \GL_r, g_{ij})$ for $\cP$ over some fppf covering $(T_i \to T)$.
For $V$ the standard $r$-dimensional representation of $\GL_r$, we have that the twist 
$\twist \cP V$ is a module in $B_r(T)$. By construction, $\twist \cP V$ has a descent datum   
$(\cO_{T_i}^r, g_{ij})$. By comparing descent data, we deduce that 
$\cE \mapsto \Isom_T(\cO^r_T, \cE)$ is an isomorphism with inverse morphism 
$\cP \mapsto \twist \cP V$. 
\end{proof}
Note that Lemma~\ref{GL_r-bundle} implies that any $\GL_r$-bundle is Zariski
locally trivial. We will henceforth implicitly use the isomorphism $B_r \simeq B\GL_r$ to
pass between locally free modules and $\GL_r$-bundles. 
\medskip

Fix a base $k$-scheme $S$ and let $p:X \to S$ be a flat, strongly projective morphism.
Fix a very ample invertible sheaf $\cO(1)$ on $X$. Lemma~\ref{GL_r-bundle} 
implies that $\Bun_{\GL_r} \simeq \Bun_r$, where $\Bun_r(T)$ is the groupoid
of locally free $\cO_{X_T}$-modules of rank $r$.  
We say that an $\cO_{X_T}$-module $\cF$ is \emph{relatively generated by global sections} if the 
counit of adjunction $p_T^*p_{T*} \cF \to \cF$ is surjective. 

Let $\cF$ be a quasi-coherent sheaf on $X_T$, flat over $T$. We say that $\cF$ is 
\emph{cohomologically flat over $T$ in degree $i$} (see \cite[8.3.10]{FGA-explained}) if for 
any Cartesian square 
\begin{equation}\label{cartesian-square}
\xymatrix{ X_{T'} \ar[r]^v \ar[d]_{p_{T'}} & X_T \ar[d]^{p_T} \\ T' \ar[r]^u & T } 
\end{equation}
the canonical morphism $u^* R^ip_{T*} \cF \to R^ip_{T'*}(v^* \cF)$
from \cite[8.2.19.3]{FGA-explained} is an isomorphism. 

\begin{lem}\label{cohom-flat}
For an $S$-scheme $T$, let $\cF$ be a quasi-coherent sheaf on $X_T$, flat over $T$. 
If $p_{T*} \cF$ is flat and $R^i p_{T*} \cF=0$ for $i>0$, then $\cF$ is 
cohomologically flat over $T$ in all degrees. For this lemma, we only need $p_T$ to be separated. 
\end{lem}
\begin{proof}
Consider the Cartesian square \eqref{cartesian-square}. To check that
$u^* R^ip_{T*} \cF \to R^ip_{T'*}(v^* \cF)$ is an isomorphism, 
it suffices to restrict to an affine subset of $T'$ lying over an affine subset of $T$. 
By the base change formula \cite[Theorem 8.3.2]{FGA-explained}, we have a canonical
quasi-isomorphism 
\[ Lu^* Rp_{T*} \cF \oarrow\sim Rp_{T'*}(v^*\cF) \]
in the derived category $D(T')$ of $\cO_{T'}$-modules. 
Since $R^ip_{T*}\cF = 0$ for $i>0$, truncation gives
a quasi-isomorphism $p_{T*}\cF \simeq Rp_{T*}\cF$. By flatness of $p_{T*}\cF$, we have
$Lu^*p_{T*}\cF \simeq u^*p_{T*}\cF$. Therefore $u^*p_{T*}\cF \simeq Rp_{T'*}(v^*\cF)$ 
in the derived category, and applying $H^i$ gives \[u^*R^ip_{T*}\cF \simeq R^ip_{T'*} (v^*\cF),\]
which is the canonical morphism of \cite[8.3.2.3]{FGA-explained}. By 
\cite[Theorem 8.3.2]{FGA-explained}, this morphism is equal to the base change morphism of 
\cite[8.2.19.3]{FGA-explained}.
\end{proof}

\begin{prop} \label{U_n-functor}
For an $S$-scheme $T$, let
\[ \cU_n(T) = \left\{ \begin{array}{c} \cE \in \Bun_r(T)  \mid R^i p_{T*}(\cE(n)) =0 
\text{ for all } i >0, \\ 
\cE(n) \text{ is relatively generated by global sections} 
\end{array}\right\} \] 
be the full subgroupoid of $\Bun_r(T)$. For $\cE \in \cU_n(T)$, the direct image 
$p_{T*}(\cE(n))$ is flat, and $\cE(n)$ is cohomologically flat over $T$ in all
degrees. In particular, the inclusion $\cU_n \into \Bun_r$ makes $\cU_n$ a pseudo-functor.
\end{prop}
\begin{proof}
For $\cE \in \cU_n(T)$, let $\cF = \cE(n)$. To prove $p_{T*}\cF$ is flat, we may 
restrict to an open affine of $T$. Assuming $T$ is affine, $X_T$ is quasi-compact and separated. 
Thus we can choose a finite affine cover $\fU = (U_j)_{j=1,\dots,N}$ of $X_T$. 
Since $X_T$ is separated, \cite[Proposition 1.4.1]{EGA3a} implies there is a canonical
isomorphism
\[ \check H^i(\fU, \cF) \overset\sim\to H^i(X_T,\cF). \] 
Since $R^ip_{T*}\cF$ is the quasi-coherent sheaf associated to $H^i(X_T,\cF)$ by 
\cite[Proposition 1.4.10, Corollaire 1.4.11]{EGA3a},
we deduce that $\check H^i(\fU,\cF)=0$ for $i>0$. Therefore we have an exact sequence 
\[ 0 \to \Gamma(X_T,\cF) \to \check C^0(\fU,\cF) \to \dots \to \check C^{N-1}(\fU,\cF) \to 0. \]
Since $\cF$ is $T$-flat, $\check C^i(\fU,\cF) = \prod_{j_0 < \dots < j_i} \Gamma(U_{j_0} \cap 
\dots \cap U_{j_i},\cF)$ is also $T$-flat. By induction, 
we conclude that $\Gamma(X_T,\cF)$ is flat on $T$. Therefore $p_{T*}\cF$ is $T$-flat. 
Consequently, Lemma~\ref{cohom-flat} implies that $\cF$ is cohomologically flat over $T$ in all
degrees.

Now we wish to show that for a $T$-scheme $T'$, the pullback $\cF_{T'}$
is relatively generated by global sections. This property is local on the base, 
so we may assume $T$ and $T'$ are affine. In this case, there exists a surjection 
$\bigoplus \cO_{X_T} \onto \cF$, so pulling back gives a surjection $\bigoplus \cO_{X_{T'}}
\onto \cF_{T'}$, which shows that $\cF_{T'}$ is generated by global sections.
We conclude that $\cE_{T'} \in \cU_n(T')$, so $\cU_n$ is a pseudo-functor. 
\end{proof}

\begin{rem} \label{noetherian-base}
Suppose $S$ is affine. Then $\Gamma(S,\cO_S)$ is an inductive limit of
finitely presented $k$-algebras, so by \cite[\S 8]{EGA4c} there exists 
an affine scheme $S_1$ of finite presentation over $\spec k$, a morphism $S \to S_1$,
and a flat projective morphism $p_1: X_1 \to S_1$ which is equal to $p$ after base change.
We can then define all of our pseudo-functors with respect to $p_1 : X_1 \to S_1$
from $(\Sch_{/S_1})^\op \to \Gpd$. The corresponding pseudo-functors for $p$ are then 
obtained by the restriction $\Sch_{/S} \to \Sch_{/S_1}$. This will allow us to reduce 
to the case where the base scheme $S$ is noetherian.
\end{rem}

\begin{lem} \label{U_n-opencover} 
The pseudo-subfunctors $(\cU_n)_{n \in \bZ}$ form an open cover of $\Bun_r$. 
\end{lem}
\begin{proof}
Take an $S$-scheme $T$ and a locally free $\cO_{X_T}$-module $\cE \in \Bun_r(T)$. 
Then the $2$-fibered product $\cU_n \xt_{\Bun_r} T : (\Sch_{/T})^\op \to \Gpd$ sends $T' \to T$ 
to the equivalence relation\footnote{By an equivalence relation, we mean a groupoid
whose only automorphisms are the identities.}
\[ \bigl\{ (\cE' \in \cU_n(T'), \cE' \simeq \cE_{T'}) \bigr\} + \mbox{isomorphisms}. \]
Therefore we must prove that for each $n$, there exists an open subscheme $U_n \subset T$ with
the universal property that $T' \to T$ factors through $U_n$ if and only if
$R^ip_{T'*}(\cE_{T'}(n))=0$ in degrees $i>0$ and $\cE_{T'}(n)$ is relatively generated
by global sections over $T'$. 

The assertions are Zariski local on $S$, so
by Remark~\ref{noetherian-base} we can assume that $S$ is the spectrum of a 
finitely generated $k$-algebra (and hence an affine noetherian scheme). 

We can assume $T$ is affine and express it as a projective limit $\ilim T_\lambda$ of spectra 
of finitely generated $\Gamma(S,\cO_S)$-algebras. Since $X \to S$ is proper, $X$ is 
quasi-compact and
separated. The fibered products $X_{T_\lambda}$ are affine over $X$. 
Therefore applying \cite[Th\'eor\`eme 8.5.2(ii), Proposition 8.5.5]{EGA4c} to
$X_T = \ilim X_{T_\lambda}$, we find that there exists $\lambda$ and 
a locally free $\cO_{X_{T_\lambda}}$-module $\cE_\lambda$ that pulls back to $\cE$. 
Hence $T \to \Bun_r$ factors through $T_\lambda$. Replacing $T$ with $T_\lambda$, we
can assume that $T$ is affine and noetherian. 

Set $\cF = \cE(n)$. Let $U_n \subset T$ be the subset of points $t \in T$ where 
$H^i(X_t, \cF_t) =0$ for $i>0$ and $\Gamma(X_t,\cF_t) \ot \cO_{X_t} \onto \cF_t$ is surjective. 
By coherence of $\cF$ and Nakayama's lemma, the set of points where $\Gamma(X_t,\cF_t) \ot 
\cO_{X_t} \onto \cF_t$ forms an open subset of $T$.
Since $X$ is quasi-compact, it can be covered
by $N$ affines. Base changing, we have that $X_t$ can also be covered by $N$ affines. 
Since $X_t$ is quasi-compact and separated, $H^i(X_t, \cF_t)$ can be computed
using \v Cech cohomology, from which we see that $H^i(X_t,\cF_t)=0$ for $i\ge N$ and all 
$t \in T$. For fixed $i$, the set of $t \in T$ where $H^i(X_t,\cF_t)=0$ is an open subset
by \cite[III, Theorem 12.8]{Hart}. By intersecting finitely many open subsets, we 
conclude that $U_n$ is an open subset of $T$. 
 For $t \in U_n$, we have by \cite[III, Theorem 12.11]{Hart} that 
$R^i p_{T*}\cF \ot \kappa(t) =0$ for $i>0$. Since $p$ is proper, $R^ip_*\cF$ is 
coherent \cite[Th\'eor\`eme 3.2.1]{EGA3a}. By Nakayama's lemma, this implies that
\[ R^i p_{U_n *}(\cF_{U_n}) \simeq (R^ip_{T*}\cF)|_{U_n} = 0.\] 
We also have by Nakayama's lemma that $\cF_{U_n}$ is relatively generated by global sections
over $U_n$. 
Therefore $\cF_{U_n} \in \cU_n(U_n)$. 

Now suppose there is a morphism $u: T' \to T$ such that $R^ip_{T'*}\cF_{T'} = 0$ for $i>0$
and $\cF_{T'}$ is relatively generated by global sections. 
Take $t' \in T'$ and set $t = u(t')$. By Proposition~\ref{U_n-functor}, we have that 
$H^i(X_{t'}, \cF_{t'})=0$ for $i>0$. Since $\spec \kappa(t') \to \spec \kappa(t)$ is 
faithfully flat, \cite[III, Proposition 9.3]{Hart} implies that 
\[ H^i(X_t, \cF_t) \ot_{\kappa(t)} \kappa(t') \simeq H^i(X_{t'} , \cF_{t'}) =0,\]
and therefore $H^i(X_t,\cF_t)=0$ for $i>0$. 
We also have from Proposition~\ref{U_n-functor} and faithful flatness that 
$\Gamma(X_{t'},\cF_{t'}) \ot \cO_{X_{t'}} \onto \cF_{t'}$ implies 
$\Gamma(X_t,\cF_t) \ot \cO_{X_t} \onto \cF_t$. 
This proves that $u$ factors through $U_n$.

By \cite[III, Theorem 5.2(b)]{Hart}, there 
exists an $n \in \bZ$ such that $R^ip_{T*}(\cE(n))=0$ for $i >0$ and $\cE(n)$ is generated
by global sections. Therefore the 
collection $(U_n)_{n \in \bZ}$ form an open cover of $T$.
\end{proof}

\begin{cor} \label{factor-noetherian-U_n} Suppose $S$ is affine 
and $T$ is an affine scheme mapping to $\cU_n$. Then $T \to \cU_n$ factors through a 
scheme locally of finite type over $S$. \end{cor}
\begin{proof}
From the proof of Lemma~\ref{U_n-opencover}, the composition $T \to \Bun_r$ factors through a 
scheme $T_1$ of finite type over $S$. Thus $T \to \cU_n$ factors through 
$\cU_n \xt_{\Bun_r} T_1$, which is representable by an open subscheme
of $T_1$ by Lemma~\ref{U_n-opencover}. 
\end{proof}

\begin{rem} \label{direct-image-free} 
For an $S$-scheme $T$ and $\cE \in \cU_n(T)$, we claim that the direct image 
$p_{T*}(\cE(n))$ is a locally free $\cO_T$-module of finite rank. 
The claim is Zariski local on $T$, so we can assume both $S$ and $T$ are affine. 
We reduce to the case where $S$ is noetherian by Remark~\ref{noetherian-base}. 
By Corollary~\ref{factor-noetherian-U_n}, there exists a locally noetherian $S$-scheme $T_1$,
a Cartesian square 
\[ \xymatrix{ X_T \ar[r]^v \ar[d]_{p_T} & X_{T_1} \ar[d]^{p_{T_1}} \\ T \ar[r]^u & T_1 } \]
and $\cE_1 \in \cU_n(T_1)$ such that $v^*\cE_1 \simeq \cE$. 
The direct image $p_{T_1*}(\cE_1(n))$ is flat and coherent by Proposition~\ref{U_n-functor}
and \cite[Th\'eor\`eme 3.2.1]{EGA3a}, hence locally free of finite rank.
By cohomological flatness of $\cE_1(n)$, we have 
$p_{T*}(\cE(n)) \simeq u^*p_{T_1*}(\cE_1(n))$, which proves the claim.
\end{rem}

For a polynomial $\Phi \in \bQ[\lambda]$, define a pseudo-subfunctor $\Bun_r^\Phi \subset 
\Bun_r$ by letting 
\[ \Bun_r^\Phi(T) = \bigl\{ \cE \in \Bun_r(T) \mid 
\Phi(m)=\chi(\cE_t(m))\text{ for all } t \in T, m\in \bZ \bigr\} \]
be the full subgroupoid.
For a locally noetherian $S$-scheme $T$ and $\cE \in \Bun_r(T)$, the Hilbert polynomial of 
$\cE_t$ is a locally constant function on $T$ by \cite[Th\'eor\`eme 7.9.4]{EGA3b}. 
Therefore we deduce from Remark~\ref{noetherian-base} and the proof of Lemma~\ref{U_n-opencover} 
that the
$(\Bun_r^\Phi)_{\Phi\in \bQ[\lambda]}$ form a disjoint open cover of $\Bun_r$.
Let \[ \cU_n^\Phi = \cU_n \cap \Bun_r^\Phi,\] 
so that $(\cU_n^\Phi)_{\Phi \in \bQ[\lambda]}$ give a disjoint open cover of $\cU_n$. 
For a general $S$-scheme $T$ and $\cE \in \cU_n^\Phi(T)$, the direct image $p_{T*}(\cE(n))$
is locally free of finite rank by Remark~\ref{direct-image-free}. By cohomological 
flatness of $\cE(n)$ and the assumption that
$R^ip_{T*}(\cE(n))=0$ for $i>0$, we find that $H^0(X_t, \cE_t(n)) \simeq 
p_{T*}(\cE(n)) \ot \kappa(t)$ is a vector space of dimension $\Phi(n)$. Hence 
$p_{T*}(\cE(n))$ is locally free of rank $\Phi(n)$. 

Since $\Bun_r$ is an fpqc stack, \cite[Lemma \href{http://math.columbia.edu/algebraic_geometry/stacks-git/locate.php?tag=05UN}{05UN}]{stacks-project} shows that all the pseudo-functors 
defined in the previous paragraph are in fact also fpqc stacks.
\medskip

Our goal now is to find schemes with smooth surjective morphisms to the $\cU_n^\Phi$. 
To do this, we will introduce a few additional pseudo-functors. 
\medskip

We define the pseudo-functors $\cY_n^\Phi : (\Sch_{/S})^\op \to \Gpd$. 
For an $S$-scheme $T$, let 
\[ \cY_n^\Phi(T) = \left\{ \begin{array}{c} (\cE, \phi, \psi) \mid 
\cE \in \cU_n^\Phi(T),\, \mbox{$\phi:\cO_{X_T}^{\Phi(n)} \onto \cE(n)$ is surjective}, \\
\mbox{the adjoint morphism $\psi:\cO_T^{\Phi(n)} \to p_{T*}(\cE(n))$ is an isomorphism}
\end{array} \right\}, \]
where a morphism $(\cE,\phi,\psi) \to (\cE',\phi',\psi')$ is an isomorphism $f: \cE \to\cE'$ 
satisfying $f_n \circ \phi = \phi'$, for $f_n = f \ot \id_{\cO(n)} : \cE(n) \simeq \cE'(n)$.
The latter condition is equivalent to $p_{T*}(f_n) \circ \psi = \psi'$, by 
adjunction. An isomorphism $\psi$ is the same as specifying $\Phi(n)$ elements of 
$\Gamma(X_T, \cE(n))$ that form a basis of $p_{T*}(\cE(n))$ as an $\cO_T$-module. Note that
$\cY_n^\Phi(T)$ is an equivalence relation. 


\begin{lem} \label{Y_n-compatible}
Suppose we have a Cartesian square
\[ \xymatrix{ X_{T'} \ar[r]^v \ar[d]_{p_{T'}} & X_T \ar[d]^{p_T} \\ T' \ar[r]^u & T } \]
Let $\cM$ be an $\cO_T$-module, $\cN$ an $\cO_{X_T}$-module, and $\phi : p_T^*\cM \to \cN$
a morphism of $\cO_{X_T}$-modules. If $\psi : \cM \to p_{T*}\cN$ is the adjoint morphism, then
the composition of $u^*(\psi)$ with the base change morphism $u^*p_{T*}\cN \to p_{T'*}v^*\cN$
corresponds via adjunction of $p_{T'*}, p_{T'}^*$ to $v^*(\phi) : p_{T'}^* u^* \cM 
\simeq v^*p_T^*\cM \to v^*\cN$. 
\end{lem}
\begin{proof}
Take $v^*(\phi) : v^*p_T^*\cM \to v^*\cN$. The isomorphism $p_{T'}^* u^* \cM \simeq 
v^*p_T^* \cM$ follows by adjunction from the equality $u_* p_{T'*} = p_{T*} v_*$.
Therefore it suffices to show that the two morphisms in question are equal after adjunction 
as morphisms $\cM \to p_{T*}v_* v^* \cN = u_* p_{T'*} v^*\cN$. On the one hand, 
the adjoint of $v^*(\phi)$ under $v_*,v^*$ is the composition $p_T^*\cM \to v_*v^*p_T^*\cM
\to v_*v^*\cN$. By naturality of the unit of adjunction, this morphism is equal to
the composition $p_T^*\cM \to \cN \to v_*v^*\cN$. By adjunction under $p_{T*}, p_T^*$, this
morphism corresponds to \[ \cM \to p_{T*} \cN \to p_{T*} v_* v^*\cN. \]
On the other hand, the adjoint of $u^*\cM \to u^*p_{T*} \cN \to p_{T'*} v^*\cN$ under
$u_*,u^*$ is equal to $\cM \to p_{T*} \cN \to u_* p_{T'*} v^*\cN = p_{T*} v_* v^*\cN$. 
By definition of the base change morphism \cite[8.2.19.2-3]{FGA-explained}, the 
latter morphism $p_{T*} \cN \to p_{T*} v_* v^*\cN$ is $p_{T*}$ applied to the unit morphism.
Hence the two morphisms under consideration are the same.
\end{proof}

Letting $\cM = \cO_T^{\Phi(n)}$ and $\cN = \cE(n)$ in Lemma~\ref{Y_n-compatible}, we see that
the pullbacks of $\phi$ and $\psi$ are compatible. Thus $\cY_n^\Phi$ has the structure of a 
pseudo-functor in an unambiguous way.

\begin{lem} \label{W-rep}
For $\Psi \in \bQ[\lambda]$, define the pseudo-functor $\cW^\Psi : (\Sch_{/S})^\op \to \Gpd$ by 
\[ \cW^\Psi(T) = \bigl\{ (\cF,\phi) \mid \cF \in \Bun^\Psi_r(T),\,
\phi : \cO_{X_T}^{\Psi(0)} \onto \cF \bigr\} \]
where a morphism $(\cF,\phi) \to (\cF',\phi')$ is an isomorphism $f : \cF \to \cF'$ 
satisfying $f \circ \phi  =\phi'$. Then $\cW^\Psi$ is representable by a strongly
quasi-projective $S$-scheme. 
\end{lem}
\begin{proof}
Note that the $\cW^\Psi(T)$ are equivalence relations, so by considering sets of equivalence
classes, $\cW^\Psi$ is isomorphic to a subfunctor \[ W^\Psi \subset
\Quot^{\Psi,\cO(1)}_{\cO_X^{\Psi(0)}/X/S} =: Q. \]
\cite[Theorem 2.6]{AK80} shows that $Q$ is representable by a scheme strongly projective 
over $S$. We show that $W^\Psi \into Q$ is schematic, open, and finitely presented, which
will imply that $\cW^\Psi \simeq W^\Psi$ is representable by a strongly quasi-projective 
$S$-scheme.
The claim is Zariski local on $S$, so by \cite[0, Corollaire 4.5.5]{EGA1-2ed} and 
Remark~\ref{noetherian-base}, we can assume $S$ is noetherian.

Let $T$ be an $S$-scheme and $(\cF,\phi) \in Q(T)$. Let $U \subset X_T$ 
be the set of points $x$ where the stalk $\cF \ot \cO_{X_T,x}$ is free. 
Since $\cF$ is coherent, $U$ is an open subset. In particular, it is the maximal open 
subset such that $\cF|_U$ is locally free. We claim that the open subset 
\[ T - p_T(X_T-U) \subset T\] 
represents the fibered product $W^\Psi \xt_Q T$. This is equivalent to
saying that a morphism $T' \to T$ lands in $T - p_T(X_T-U)$ if and only if $\cF_{T'}$ is locally free.
If $T' \to T$ lands in $T - p_T(X_T-U)$, then $X_{T'} \to X_T$ lands in $U$, which implies
$\cF_{T'}$ is locally free. Conversely, suppose we start with a morphism $T' \to T$ such that
$\cF_{T'}$ is locally free. Assume for the sake of contradiction that there exists $t' \in T$ and 
$x \in X_T$ that morphism to the same point $t \in T$ such that $\cF \ot \cO_{X_T,x}$ is not free.
We have Cartesian squares
\[ \xymatrix{ (X_{T'})_{t'} \ar[r] \ar[d] & (X_T)_t \ar[d]\ar[r] & X_T \ar[d]  \\ 
\spec \kappa(t')\ar[r] & \spec \kappa(t) \ar[r] & T } \]
where $\spec \kappa(t') \to \spec \kappa(t)$ is faithfully flat. 
Then $\cF_{t'}\simeq \cF_t \ot_{\kappa(t)} 
\kappa(t')$ is a flat $\cO_{X_{t'}}$-module. By faithful flatness, this implies that
$\cF_t$ is flat over $\cO_{X_t}$. Recall from the definition of $Q$ that
$\cF$ is assumed to be $T$-flat. Therefore $\cF \ot \cO_{X_T,x}$ is $\cO_{T,t}$-flat and 
$\cF_t \ot \cO_{X_t,x}$ is $\cO_{X_t,x}$-flat. By \cite[20.G]{Matsumura}, we conclude that
$\cF \ot \cO_{X_T,x}$ is $\cO_{X_T,x}$-flat and hence a free module.
We therefore have a contradiction, so $T' \to T$ must factor through $T - p_T(X_T-U)$.
Letting $T = Q$, we have shown that $W^\Psi$ is representable by an open subscheme 
of $Q$. Since $S$ is noetherian, $Q$ is also noetherian. Therefore $W^\Psi \into Q$
is finitely presented.
\end{proof}

\begin{lem} \label{Y_n-rep}
The pseudo-functor $\cY_n^\Phi$ is representable by a scheme $Y_n^\Phi$ which is
strongly quasi-projective over $S$. 
\end{lem}
\begin{proof} 
Define the polynomial $\Psi \in \bQ[\lambda]$ by $\Psi(\lambda) = \Phi(\lambda+n)$. 
Let $\cW^\Psi$ be the pseudo-functor of Lemma~\ref{W-rep}. Then there is a morphism 
$\cW^\Psi \to \Bun^\Phi_r$ by sending $(\cF,\phi) \mapsto \cF(-n)$. The corresponding
$2$-fibered product $\cU_n^\Phi \xt_{\Bun^\Phi_r} \cW^\Psi$ is isomorphic to the pseudo-functor 
$\cZ_n^\Phi : (\Sch_{/S})^\op \to \Gpd$ defined by
\[ \cZ_n^\Phi(T) = \bigl\{ (\cE, \phi) \mid \cE \in \cU_n^\Phi(T), \phi : \cO^{\Phi(n)}_{X_T} \onto 
\cE(n) \bigr\} \]
where a morphism $(\cE, \phi) \to (\cE',\phi')$ is an isomorphism $f : \cE \to \cE'$ satisfying
$f_n \circ \phi = \phi'$. By Lemma~\ref{U_n-opencover}, the morphism
$\cZ_n^\Phi \to \cW^\Psi$ is schematic and open, so Lemma~\ref{W-rep} implies that $\cZ_n^\Phi$
is representable by an open subscheme $Z_n^\Phi$ of $W^\Psi$. We claim that $Z_n^\Phi 
\into W^\Psi$ is finitely presented. By Remark~\ref{noetherian-base}, we can reduce to
the case where $S$ is noetherian, and now the assertion is trivial. Therefore
$Z_n^\Phi$ is strongly quasi-projective over $S$.

Next we show that the morphism $\cY_n^\Phi \to \cZ_n^\Phi$ sending 
$(\cE,\phi,\psi) \mapsto (\cE,\phi)$
is schematic, open, and finitely presented. By \cite[0, Corollaire 4.5.5]{EGA1-2ed} and
Remark~\ref{noetherian-base}, we assume $S$ is noetherian. 
Take $(\cE,\phi) \in \cZ_n^\Phi(T)$ and let $\psi$ be the adjoint morphism
$\cO_T^{\Phi(n)} \to p_{T*}(\cE(n))$. Define the open subscheme $U \subset T$ to be the 
complement of the support of $\coker(\psi)$.
By Nakayama's lemma and faithful flatness of field extensions, a morphism $u:T' \to T$
lands in $U$ if and only if the pullback $u^*(\psi)$ is a surjection. 
Since $\cO_T^{\Phi(n)}$ and $p_{T*}(\cE(n))$ are locally free $\cO_T$-modules of the same
rank, $u^*(\psi)$ is surjective if and only if it is an isomorphism.
Recall from Proposition~\ref{U_n-functor} that the base change morphism $u^*p_{T*}(\cE(n)) \to
p_{T'*}(\cE_{T'}(n))$ is an isomorphism. Thus the compatibility assertion of 
Lemma~\ref{Y_n-compatible} shows that $u^*(\psi)$ is an isomorphism if and only if
the adjoint of $v^*(\phi)$ is an isomorphism, for $v : X_{T'} \to X_T$. 
Thus $U$ represents the fibered product $\cY_n^{\Phi} \xt_{\cZ_n^\Phi} T$. 
Taking $T = Z_n^\Phi$, which is noetherian, we deduce that $\cY_n^\Phi \to \cZ_n^\Phi$ 
is schematic, open, and finitely presented. Therefore $\cY_n^\Phi$ is representable
by a strongly quasi-projective $S$-scheme.
\end{proof}

\begin{lem} \label{Y_n-U_n} 
There is a canonical right $\GL_{\Phi(n)}$-action on $Y_n^\Phi$ such that 
the morphism $Y_n^\Phi \to \cU_n^\Phi$ sending 
$(\cE,\phi,\psi) \mapsto \cE$ is a $\GL_{\Phi(n)}$-bundle. 
Lemma~\ref{Z/G=X} gives an isomorphism 
\[ \cU_n^\Phi \simeq [Y_n^\Phi/\GL_{\Phi(n)}] \]
sending $\cE \in \cU_n^\Phi(T)$ to a $\GL_{\Phi(n)}$-equivariant
morphism $\Isom_T\bigl(\cO_T^{\Phi(n)}, p_{T*}(\cE(n))\bigr) \to Y_n^\Phi$.
\end{lem}
\begin{proof} We define an action morphism 
$\alpha:\cY_n^\Phi \xt \GL_{\Phi(n)} \to \cY_n^\Phi$ by sending 
\[ ((\cE,\phi,\psi), g) \mapsto (\cE,\phi \circ p_T^*(g),\psi \circ g) \] 
over an $S$-scheme $T$.  Then the diagram 
\[ \xymatrix{ \cY_n^\Phi \xt \GL_{\Phi(n)} \ar[r]^-\alpha \ar[d]_{\pr_1} & \cY_n^\Phi \ar[d] \\
\cY_n^\Phi \ar[r] & \cU_n^\Phi } \]
is $2$-commutative where the $2$-morphism is just the identity. Then $\alpha$
induces a right $\GL_{\Phi(n)}$-action on $Y_n^\Phi$ by taking equivalence classes, i.e., 
we have a $2$-commutative square
\[ \xymatrix{ Y_n^\Phi \xt \GL_{\Phi(n)} \ar[r] \ar[d]_\sim & Y_n^\Phi \ar[d]^\sim \\ 
\cY_n^\Phi \xt \GL_{\Phi(n)} \ar[r]^-\alpha & \cY_n^\Phi } \]
and the associated $2$-morphism must satisfy the associativity condition of 
\ref{sect:G-invariant}(2) because $\cY_n^\Phi(T)$ is an equivalence relation. 
We deduce that $Y_n^\Phi \to \cU_n^\Phi$ is $\GL_{\Phi(n)}$-invariant.

For an $S$-scheme $T$, let $\cE \in \cU_n^\Phi(T)$. Then the fibered product 
on $(\Sch_{/T})^\op \to \Gpd$ is given by 
\[ \Bigl(\cY_n^\Phi \xt_{\cU_n^\Phi} T \Bigr)(T') = \bigl\{ (\cN, \phi, \psi, \gamma) \mid 
(\cN,\phi,\psi) \in \cY_n^\Phi, \gamma : \cN \simeq \cE_{T'}  \bigr\} \] 
where a morphism $(\cN, \phi,\psi,\gamma) \to (\cN',\phi',\psi',\gamma')$ is an isomorphism
$f : \cN \to \cN'$ such that $f_n\circ \phi = \phi'$ and $\gamma' \circ f = \gamma$. 
Observe that $(\cY_n^\Phi \xt_{\cU_n^\Phi} T)(T')$ is an equivalence relation, and 
$\cY_n^\Phi \xt_{\cU_n^\Phi} T$ is isomorphic to the functor $\cP : (\Sch_{/T})^\op \to \Set$ sending
\[ (T' \to T) \mapsto 
\bigl\{ (\phi,\psi) \mid \phi : \cO_{X_{T'}}^{\Phi(n)} \onto \cE_{T'}(n),\, \text{the adjoint } 
\psi \text{ is an isomorphism} \bigr\}. \]
In fact, the surjectivity of $\phi$ is automatic:
since $\cE_{T'}(n)$ is relatively generated by global sections, 
if we have an isomorphism $\psi : \cO_{T'}^{\Phi(n)} \simeq p_{T'*}(\cE_{T'}(n))$, then the 
adjoint morphism 
\[ \phi : \cO_{X_{T'}}^{\Phi(n)} \simeq p_{T'}^*p_{T'*}(\cE_{T'}(n)) \onto \cE_{T'}(n)\] 
is surjective. 
Therefore $\cP(T')$ is just equal to a choice of basis for $p_{T'*}(\cE_{T'}(n))$. 
The induced right action of $\GL_{\Phi(n)}$ on $\cP$ is defined by $\psi.g = \psi \circ g$ 
for $g \in \GL_{\Phi(n)}(T')$. By Remark~\ref{direct-image-free},
the direct image $p_{T*}(\cE(n))$ is locally free of rank $\Phi(n)$. 
By cohomological flatness of $\cE(n)$ over $T$ (see Proposition~\ref{U_n-functor}), 
we have an isomorphism $\cP \simeq \Isom_T \bigl(\cO_T^{\Phi(n)},p_{T*}(\cE(n))\bigr)$. 
The latter is the right $\GL_{\Phi(n)}$-bundle associated to $p_{T*}(\cE(n))$ 
by Lemma~\ref{GL_r-bundle}. 
\end{proof}

In fact, we can say a little more about the structure of $Y_n^\Phi$, and this
will be useful in \S\ref{section:level-structure}. 

\begin{lem} \label{Y_n-equivariant-line-bundle}
There exists a $\GL_{\Phi(n)}$-equivariant invertible sheaf on $Y_n^\Phi$ that is 
very ample over $S$.
\end{lem}
\begin{proof}
Define the polynomial $\Psi(\lambda) = \Phi(\lambda+n)$ and let $Q = 
\Quot_{\cO_X^{\Phi(n)}/X/S}^{\Psi,\cO(1)}$. By \cite[Theorem 2.6]{AK80}, 
there exists an integer $m$ and a closed immersion 
\[ Q \into \bP \Biggl( \bigwedge^{\Psi(m)} p_*(\cO_X^{\Phi(n)}(m)) \Biggr). \]
Let $\cO_Q(1)$ be the very ample invertible $\cO_Q$-module corresponding to this immersion. 
\cite[Theorem 2.6]{AK80} also implies that for a morphism
$q : T \to Q$ corresponding to $\cO_{X_T}^{\Phi(n)} \onto \cF$, we have 
a canonical isomorphism $q^*(\cO_Q(1)) \simeq \bigwedge^{\Psi(m)} p_{T*}(\cF(m))$.
From Lemmas~\ref{W-rep} and \ref{Y_n-rep}, we have an immersion 
\[ Y_n^\Phi \into Q : (\cE, \phi, \psi) \mapsto (\phi : \cO_{X_T}^{\Phi(n)} \onto \cE(n)). \]
Thus for any morphism $y : T \to Y_n^\Phi$ corresponding to $(\cE,\phi,\psi)$, 
there exists a canonical isomorphism $y^*(\cO_Q(1)|_{Y_n^\Phi}) \simeq \bigwedge^{\Phi(n+m)}
p_{T*}(\cE(n+m))$. Therefore $y^*(\cO_Q(1)|_{Y_n^\Phi})$ only depends on 
the image of $y$ in $\cU_n^\Phi(T)$. Since $Y_n^\Phi \to \cU_n^\Phi$ is 
$\GL_{\Phi(n)}$-invariant, it follows that for any $g \in \GL_{\Phi(n)}(T)$, there
is a canonical isomorphism $y^*(\cO_Q(1)|_{Y_n^\Phi}) \simeq (y.g)^*(\cO_Q(1)|_{Y_n^\Phi})$.
This gives the very ample invertible sheaf $\cO_Q(1)|_{Y_n^\Phi}$ a 
$\GL_{\Phi(n)}$-equivariant structure.
\end{proof}
\medskip

\begin{proof}[Proof of Theorem~\ref{BunG-alg}]
We know the diagonal of $\Bun_G$ is schematic, affine, and finitely presented by 
Corollary~\ref{BunG-diagonal}. For any affine algebraic group $G$, there exists an integer $r$ 
such that $G \subset \GL_r$ is a closed subgroup \cite[II, \S 2, Th\'eor\`eme 3.3]{De-Ga}. 

First, we prove the theorem for $\GL_r$. 
Lemma~\ref{Y_n-U_n} implies that the morphisms $\cY_n^\Phi \to \cU_n^\Phi$ are smooth and 
surjective. By taking an open covering of $S$ and using \cite[0, Corollaire 4.5.5]{EGA1-2ed}
and Lemma~\ref{Y_n-rep}, we see that each $\cY_n^\Phi$ is representable
by a scheme $Y_n^\Phi$ of finite presentation over $S$. The morphism 
\[ Y=\bigsqcup_{n\in \bZ, \Phi \in \bQ[\lambda]} Y_n^\Phi \to \Bun_r \]
is then smooth and surjective by Lemma~\ref{U_n-opencover}, where $Y$ is locally of finite 
presentation over $S$. Therefore $\Bun_r$ is an algebraic stack locally of finite presentation
over $S$. It follows from \cite[Lemma \href{http://math.columbia.edu/algebraic_geometry/stacks-git/locate.php?tag=05UN}{05UN}]{stacks-project} that all $\cU_n^\Phi \subset \Bun_r$ are algebraic
stacks, and they are of finite presentation over $S$ since the $Y_n^\Phi$ are.

By Corollary~\ref{BunG-cog}, the morphism $\Bun_G \to \Bun_r$ corresponding to
$G \subset \GL_r$ is schematic and locally of finite presentation. 
Consider the Cartesian squares
\[ \xymatrix{ \wtilde Y_n^\Phi \ar[r] \ar[d] & \wtilde\cU_n^\Phi \ar@{^(->}[r] \ar[d]& \Bun_G 
\ar[d] \\ Y_n^\Phi \ar[r] & \cU_n^\Phi \ar@{^(->}[r] & \Bun_r } \]
where by base change, the morphisms $\wtilde Y_n^\Phi \to \wtilde \cU_n^\Phi$ are smooth and
surjective, and the $\wtilde \cU_n^\Phi$ form an open covering of $\Bun_G$.
From Theorem~\ref{sect-general} and Corollary~\ref{BunG-cog}, we deduce that
$\wtilde Y_n^\Phi$ is representable by a disjoint union of schemes of finite
presentation over $S$. Applying \cite[Lemma \href{http://math.columbia.edu/algebraic_geometry/stacks-git/locate.php?tag=05UP}{05UP}]{stacks-project} to the smooth surjective morphism 
$\wtilde Y_n^\Phi \to \wtilde \cU_n^\Phi$, we conclude that 
$\Bun_G$ is covered by open substacks of finite presentation over $S$.
\end{proof}

\subsection{Examples} We can say a little more about the properties of $\Bun_G$
after imposing extra conditions on $G$ and $X \to S$. In this subsection we discuss
a few such examples.

\subsubsection{Case of a curve} \label{section:curve}
In the following example, assume $S = \spec k$ and $k$ is algebraically closed. 
Let $X$ be a smooth projective integral scheme of dimension $1$ over $k$. 
We mention some examples concerning the previous constructions in this situation.

First we show that for a locally free $\cO_X$-module $\cE$ of rank $r$, the
Hilbert polynomial is determined by $\deg \cE$. Let $K(X)$ denote the Grothendieck
group of $X$. By \cite[Corollary 10.8]{Manin-Ktheory}, the homomorphism
\[ K(X) \to \Pic(X) \oplus \bZ\]
sending a coherent sheaf $\cF \in \Coh(X)$ to $(\det\cF,\on{rk} \cF)$ is an isomorphism,
and the inverse homomorphism sends $(\cL,m) \mapsto [\cL] + (m-1)[\cO_X]$.
Therefore $[\cE] = [\det \cE] + (r-1)[\cO_X]$ in $K(X)$. Tensoring by $\cO(n)$, we deduce
that \[ [\cE(n)] = [(\det \cE) \ot \cO(n)] + (r-1)[\cO(n)]. \]
Applying $\deg$ to this equality, we find that $\deg \cE(n) = \deg \cE + r\deg \cO(n)$,
or equivalently $\chi(\cE(n)) = \deg \cE + r \chi(\cO(n))$. Therefore 
$\Bun_r$ is a disjoint union of $(\Bun_r^d)_{d\in \bZ}$, the substacks of locally free
sheaves of rank $r$ which are fiber-wise of degree $d$, and the previous 
results hold with $\Bun_r^\Phi$ replaced by $\Bun_r^d$. 
We also note that since $\dim X=1$, for any $\cO_X$-module $\cF$, the higher
cohomologies $H^i(X,\cF)=0$ for $i>1$ by \cite[Th\'eor\`eme 3.6.5]{Grothendieck-Tohoku}.
Therefore the substacks $\cU_n$ are characterized by relative generation by 
global sections and vanishing of $R^1 p_{T*}$. 

Theorem~\ref{BunG-alg} says that $\Bun_r$ is locally of finite type over $k$, so
the stack is a disjoint union of connected components (take the images of the
connected components of a presentation using \cite[Lemma \href{http://math.columbia.edu/algebraic_geometry/stacks-git/locate.php?tag=05UP}{05UP}]{stacks-project}). 
By considering degrees, $\Bun_r$ has infinitely many connected components, so it is not 
quasi-compact. One might ask if the connected components of $\Bun_r$ are quasi-compact. 
The following example shows that this should not be expected to be the case in general.

\begin{eg} \label{Bunr-notqc} 
There exists a connected component of $\Bun_2$ that is not quasi-compact. 
First we show that there exists $n_0$ such that $\cO_X^2$ and $\cO(-n) \oplus \cO(n)
\in \Bun_2(k)$ all lie in the same connected component of $\Bun_2$ for $n \ge n_0$.
By Serre's Theorem \cite[II, Theorem 5.17]{Hart}, there exists $n_0$ such that
$\cO_X^2(n)$ is generated by finitely many global sections for $n\ge n_0$. Fix such an $n$.
Using the fact that $\dim X=1$, it follows from a lemma of Serre 
\cite[pg.~148]{Mumford-curves} that there exists a section $s \in \Gamma(X,\cO_X^2(n))$
that is nonzero in $\cO_X^2(n) \ot \kappa(x)$ for all closed points $x \in X(k)$.
This gives a short exact sequence of $\cO_X$-modules
\[ 0 \to \cO(-n) \to \cO_X^2 \to \cL \to 0\]
where $\cL$ is locally free of rank $1$. From the definition of the determinant, 
we have \[ \cL \simeq \det \cL \simeq \bigwedge^2(\cO_X^2) \ot \cO(n) \simeq \cO(n).\]
Therefore $\cO_X^2$ is an extension of $\cO(n)$ by $\cO(-n)$. 
By \cite[Corollaire 4.2.3]{Grothendieck-Tohoku}, we have $\Ext^1_{\cO_X}(\cO(n),\cO(-n))
\simeq H^1(X, \cO(-2n))$, which is a finite dimensional $k$-vector space 
\cite[III, Theorem 5.2]{Hart}. Let 
\[ \bV = \spec \sym_k (H^1(X,\cO(-2n))^\vee).\] 
For an affine noetherian $k$-scheme $T = \spec A$, we have 
$\Hom_k(T, \bV) \simeq H^1(X,\cO(-2n)) \ot_k A$. 
By cohomology and flat base change \cite[III, Proposition 9.3]{Hart}, 
this gives \[ \Hom_k(T,\bV) \simeq H^1(X_T, \cO_{X_T}(-2n)) \simeq \Ext^1_{\cO_{X_T}}
(\cO_{X_T}(n), \cO_{X_T}(-n)) . \] 
The identity morphism $\id_\bV$ corresponds to some extension of $\cO_{X_\bV}(n)$
by $\cO_{X_\bV}(-n)$, which must be a locally free $\cO_{X_\bV}$-module of rank $2$.
This defines a morphism $\bV \to \Bun_2$. 
We have that $\cO_X^2$ and $\cO(-n) \oplus \cO(n)$ both correspond to $k$-points of $\bV$. 
Since $\bV$ is connected, this implies that these two points lie in the same connected
component of $\Bun_2$. Now suppose there exists a quasi-compact scheme $Y^\circ$ 
with a surjective morphism to the connected component of $\Bun_2$ in question. 
The morphism $Y^\circ \to \Bun_2$ corresponds to a locally free sheaf $\cE$
of rank $2$ on $X_{Y^\circ}$. 
We can assume $Y^\circ$ is noetherian by \cite[\S 8]{EGA4c}. Then 
by \cite[II, Theorem 5.17]{Hart}, there exists 
an integer $n \ge n_0$ such that $\cE(n)$ is relatively generated by global sections. 
By Nullstellensatz, there must exists a $k$-point of $Y^\circ$ mapping to the isomorphism
class of $\cO(-n-1) \oplus \cO(n+1) \in \Bun_2(k)$. Therefore we have that 
$\cO(-1) \oplus \cO(2n+1)$ is generated by global sections. In particular, $\cO(-1)$
is generated by global sections, which is a contradiction since $\Gamma(X,\cO(-1))=0$
by \cite[IV, Lemma 1.2, Corollary 3.3]{Hart}.
\end{eg}

\subsubsection{Picard scheme and stack}
In the previous example we considered locally free sheaves of rank $2$ on a dimension $1$
scheme. As the next example shows, if we consider locally free sheaves of rank $1$, then
the connected components of $\Bun_1$ will in fact be quasi-compact. 
Assume for the moment that $X \to S = \spec k$ is a curve as in \S\ref{section:curve}.

\begin{eg} Let $g$ be the genus of $X$. Then for any integer $d$, we have 
$\cU_n^d = \Bun_1^d$ for $n \ge 2g-d$. Take $Y \to \Bun_1$ a smooth surjective
morphism with $Y$ a scheme locally of finite type over $k$. It suffices to show that
$\cU_n^d \xt_{\Bun_1^d} Y \into Y$ is surjective. Since $k$ is algebraically closed,
it is enough to check surjectivity on $k$-points. Suppose we have $\cL \in \Bun_1^d(k)$. 
Then $\deg\cL(n) \ge 2g$ since $\deg \cO(1) >0$ by \cite[IV, Corollary 3.3]{Hart}. 
Therefore \cite[IV, Example 1.3.4, Corollary 3.2]{Hart} and Serre duality imply 
that $H^1(X,\cL(n))=0$ and $\cL(n)$ is generated by global sections.
We conclude that $\cU_n^d = \Bun_1^d$ for $n \ge 2g-d$, and in particular, 
$\Bun_1^d$ is quasi-compact. 
\end{eg}

Now suppose $p:X \to S$ is a separated morphism of finite type between schemes such that
the fppf sheaf $\Pic_{X/S}$ is representable by a scheme. This is satisfied, for example, 
when $S$ is locally noetherian and $p$ is flat and locally projective with geometrically
integral fibers \cite[Theorem 9.4.8]{FGA-explained}. 
Additionally assume that the unit morphism $\cO_T \to p_{T*}\cO_{X_T}$ 
is an isomorphism for all $S$-schemes $T$. 
By \cite[Exercise 9.3.11]{FGA-explained}, this holds when $S$ is locally noetherian and
$p$ is proper and flat with
geometric fibers that are reduced and connected. 

The stack $\Bun_1$ has the special property that 
the associated coarse space (see \cite[Remarque 3.19]{LMB}) is 
the Picard scheme $\Pic_{X/S}$. We will use
properties of $\Pic_{X/S}$ to deduce properties of $\Bun_1$. We thank Thanos D.~Papa\"ioannou 
for suggesting this approach. 

\begin{prop} \label{Bun1-Pic} 
Suppose $X \to S$ admits a section. Then there is an isomorphism 
\[ B\bG_m \xt \Pic_{X/S} \to \Bun_1 \] over $S$ such that the morphism to the coarse space
is the second projection. 
\end{prop}
\begin{proof} 
Since $p$ has a section, we have that 
$\Pic_{X/S}(T) = \Pic(X_T)/\Pic(T)$ by \cite[Theorem 9.2.5]{FGA-explained}.
For an invertible sheaf $\cL \in \Pic(X_T)$, we
consider the Cartesian square 
\[ \xymatrix{ \cF \ar[r] \ar[d] & \Bun_1 \ar[d] \\ T \ar[r]^-\cL & \Pic_{X/S} } \]
where $\cF : (\Sch_{/T})^\op \to \Gpd$ is the stack sending 
\[ (T' \to T) \mapsto \{ \cM \in \Pic(X_{T'}) \mid \text{there exists } \cN \in \Pic(T'),\,
\cM \simeq  \cL_{T'} \ot p_{T'}^*\cN \} \]
and a morphism $\cM \to \cM'$ is an isomorphism of $\cO_{X_{T'}}$-modules. 
We define the morphism $B\bG_m \xt T \to \cF$ by sending an invertible $\cO_{X_{T'}}$-module
$\cN$ to $\cL_{T'} \ot p_{T'}^*\cN$. Now \cite[Lemma 9.2.7]{FGA-explained} shows
that $p_{T'}^*$ is fully faithful from the category of locally free sheaves of
finite rank on $T'$ to that on $X_{T'}$. Therefore the previous morphism is an isomorphism.
Taking $T = \Pic_{X/S}$ and considering the identity morphism on $\Pic_{X/S}$ proves the claim.
\end{proof}

\begin{cor} \label{Pic-coarse}
The morphism $\Bun_1 \to \Pic_{X/S}$ is smooth, surjective, and of finite presentation. 
\end{cor}
\begin{proof}
The claim is fppf local on $S$, so we assume $S$ is connected and $X$ nonempty. 
The image of $p$ is open and closed since $p$ is flat and proper, and hence $p$ is an fppf morphism.
The base change $X \xt_S X \to X$ admits a section via the diagonal. Then
Proposition~\ref{Bun1-Pic} implies that the morphism $\Bun_1 \xt_S X \to \Pic_{X\xt_S X/X}$
is smooth, surjective, and of finite presentation because $B\bG_m \to \cdot$ 
has these properties ($\cdot \to B\bG_m$ is a smooth presentation by Lemma~\ref{BG-alg} 
since $\bG_m$ is smooth). 
\end{proof}

For $\cL \in \Bun_1(T)$, the automorphisms of $\cL$
are given by $\Gamma(X_T,\cO_{X_T})^\times \simeq \Gamma(T,\cO_T^\times)$ by our assumption on
$\cO_T \to p_{T*} \cO_{X_T}$. 
Thus $\Isom(\cL,\cL) \simeq \bG_m \xt T$. 
We note that the smoothness and surjectivity of $\Bun_1 \to \Pic_{X/S}$ then follow 
from the proof of \cite[Corollaire 10.8]{LMB}.

\begin{eg} Let $S$ be locally noetherian and $X \to S$ flat and projective with geometrically
integral fibers. Then the connected components of $\Bun_1$ are of finite type over $S$. 
Indeed, the connected components of $\Pic_{X/S}$ are of finite type over $S$ by 
\cite[Theorem 9.6.20]{FGA-explained}, so their preimages under $\Bun_1 \to \Pic_{X/S}$ 
are also of finite type over $S$ by Corollary~\ref{Pic-coarse}. 
\end{eg}

\section{Level structure} \label{section:level-structure}

Let $p : X \to S$ be a separated morphism of schemes over $k$.
Suppose we have a section $x : S \to X$ over $S$.  
Then the graph $\Gamma_x : S \to S \xt_S X$ is a closed immersion, which implies $x$ is a closed
immersion. Let $\cI_x$ be the corresponding ideal sheaf. For any positive integer $n$,
define $i_n : (nx) \into X$ to be the closed immersion corresponding to the ideal sheaf $\cI^n$.
We assume that $X$ is infinitesimally flat in $x$ (see \cite[I, 7.4]{Jantzen}), i.e.,
we require that $\pi_n : (nx) \to S$ is finite and locally free for all $n$.

\begin{rem} 
We discuss two cases where the condition of infinitesimal flatness is automatic. 
First, suppose $S$ is the spectrum of a field. Then $x$ is a closed point, and
$(nx) = \spec \cO_{X,x}/\fm_x^n$ is the spectrum of an artinian ring, where $\fm_x$ is the 
maximal ideal of $\cO_{X,x}$. Thus $(nx)$ is finite and free over $S$. 
\smallskip

Next suppose that $x$ lands in the smooth locus $X_{\on{sm}} \subset X$ of $p : X \to S$.
Pick a point $s \in S$. By assumption, $p$ is smooth at $x(s)$. Therefore by 
\cite[Th\'eor\`eme 17.12.1]{EGA4d}, there exists an open neighborhood $U$ of $s$
where $\cI_x/\cI_x^2|_U$ is a locally free $\cO_U$-module, and the canonical morphism 
$\sym^\bullet_{\cO_U}(\cI_x/\cI_x^2|_U) \to \cG r^\bullet(x|_U)$ is an isomorphism.
This is true for all $s \in S$, so we can take $U = S$ in the above. Now we have 
short exact sequences of $\cO_X$-modules
\[ 0 \to x_*(\cG r^n(x)) \to \cO_X/\cI_x^{n+1} \to \cO_X/\cI_x^n \to 0, \]
where $\cG r^n(x)$ is a locally free $\cO_S$-module. We know that $\cO_X/\cI_x \simeq 
x_*\cO_S$ is a free $\cO_S$-module. By induction we find that the above sequence
is locally split exact, and hence $\cO_X/\cI_x^n$ is a locally free $\cO_S$-module 
for all $n$. The closed immersion $S \simeq (x) \into (nx)$ is an 
isomorphism on topological spaces, 
so we deduce that $(nx) \to S$ is finite and locally free.
\end{rem}

Define the pseudo-functor $\Bun_G^{(nx)} : (\Sch_{/S})^\op \to \Gpd$ by 
\[ \Bun_G^{(nx)}(T) = \bigl\{ ( \cP, \phi ) \mid \cP \in \Bun_G(T),\, \phi \in
\Hom_{BG((nx)_T)}((nx)_T \xt G, i_n^*\cP) \bigr\} \]
where a morphism $(\cP, \phi) \to (\cP', \phi')$ is a $G$-equivariant morphism $f : \cP \to \cP'$ 
over $X_T$ such that $i_n^*(f) \circ \phi = \phi'$. 
Since $BG$ and $\Bun_G$ are fpqc stacks, we observe that $\Bun_G^{(nx)}$ is also an fpqc stack.

The main result of this section is the following:

\begin{thm} \label{compact-open-level} 
Suppose $p : X \to S$ is a flat, finitely presented, projective morphism with
geometrically integral fibers. Let $\cV \subset \Bun_G$ be a quasi-compact open substack.
There exists an integer $N_0$ such that for all $N \ge N_0$, the $2$-fibered product
$\Bun_G^{(Nx)} \xt_{\Bun_G} \cV$ is representable by a quasi-compact scheme which is
locally of finite presentation over $S$. 
\end{thm}

In the proofs that follow, we will abuse notation and use
$i_n$ and $\pi_n$ to denote the corresponding morphisms under a change of base $T \to S$.
\smallskip

The following proposition shows that Theorem~\ref{compact-open-level} is useful in
providing a smooth presentation of a quasi-compact open substack of $\Bun_G$ when 
$G$ is smooth.

\begin{prop}\label{nx}
The projection $\Bun_G^{(nx)} \to \Bun_G$ sending $(\cE,\phi)
\mapsto \cE$ is schematic, surjective, affine, and of finite presentation.
If $G$ is smooth over $k$, then the projection is also smooth. \end{prop}
\begin{proof}
For an $S$-scheme $T$, let $\cP \in \Bun_G(T)$. Then the fibered product on
$(\Sch_{/T})^\op \to \Gpd$ is given by
\[ \Bigl(\Bun_G^{(nx)} \xt_{\Bun_G} T \Bigr)(T') = \bigl\{ (\cE,\phi,\gamma) \mid 
 (\cE,\phi) \in \Bun_G^{(nx)}(T'), \gamma : \cE \to \cP_{T'} \bigr\} \]
where a morphism $(\cE,\phi,\gamma) \to (\cE',\phi',\gamma')$ is a $G$-equivariant morphism 
$f: \cE \to \cE'$ such that $i_n^*(f) \circ \phi = \phi'$ and $\gamma' \circ f = \gamma$. 
We observe that the above groupoids are equivalence relations, and 
$\Bun_r^{(nx)} \xt_{\Bun_r} T$ is isomorphic to the functor $F : (\Sch_{/T})^\op \to \Set$
sending 
\[ (T' \to T) \mapsto \Hom_{BG((nx)_{T'})}\bigl((nx)_{T'} \xt G, i_n^*(\cP_{T'})\bigr). \]
There is a canonical simply transitive right action of $\pi_{n*}G_{(nx)_T}$ on $F$ over $T$
(here the direct image $\pi_{n*}$ of 
a functor is defined as in \cite[7.6]{Bosch-neron}). 
Since $G_{(nx)_T}$ is affine and finitely presented over $(nx)_T$, 
\cite[7.6, Theorem 4, Proposition 5]{Bosch-neron} show that $\pi_{n*}G_{(nx)_T}$ is
representable by a group scheme which is affine and finitely presented over $T$. 
If $G$ is smooth over $k$, then $\pi_{n*}G_{(nx)_T}$ is also smooth over $T$.
Let $\bigl( T'_i \to (nx)_T \bigr)$ be an fppf cover trivializing $i_n^*\cP$. Then
the refinement $\bigl( T'_i \xt_T (nx)_T \to (nx)_T \bigr)$ 
also trivializes $i_n^*\cP$. Since $(nx)_T \to T$ is fppf, the compositions 
$(T'_i \to (nx)_T \to T)$ give an fppf cover trivializing $F$, which
shows that $F$ is a $\pi_{n*}G_{(nx)_T}$-torsor. We conclude that $F \to T$ has
the desired properties by descent theory.
\end{proof}

\begin{rem} Proposition~\ref{nx} shows that if $\Bun_G$ is an algebraic stack, 
then $\Bun_G^{(nx)}$ is also algebraic by \cite[Lemma \href{http://math.columbia.edu/algebraic_geometry/stacks-git/locate.php?tag=05UM}{05UM}]{stacks-project}. 
\end{rem}

For positive integers $m < n$, we have the nilpotent thickening 
$i_{m,n} : (mx) \into (nx)$. 

\begin{prop} \label{mx-nx} Suppose that $G$ is smooth over $k$. 
For positive integers $m< n$, the morphism $\Bun_G^{(nx)} \to \Bun_G^{(mx)}$ sending
$(\cP,\phi) \mapsto (\cP, i_{m,n}^*(\phi))$ is schematic, surjective, affine, and of
finite presentation. 
\end{prop}
\begin{proof} For an $S$-scheme $T$, let $(\cP, \psi) \in \Bun_G^{(mx)}(T)$. 
By similar considerations as in the proof of Proposition~\ref{nx}, we see that 
the fibered product $\Bun_G^{(nx)} \xt_{\Bun_G^{(mx)}} T$ is isomorphic to the
functor $F : (\Sch_{/T})^\op \to \Set$ sending 
\[ (T' \to T) \mapsto \bigl\{ \phi : (nx)_{T'} \xt G \to  \cP_{T'}
\mid i_{m,n}^*(\phi) = \psi_{T'} \bigr\}. \]
The pullback $i_{m,n}^*$ defines a natural morphism 
\[ \pi_{n*}G_{(nx)_T} \to \pi_{m*}G_{(mx)_T}.\] 
Let $N_{m,n}$ denote the kernel of this morphism. 
By \cite[7.6, Theorem 4, Proposition 5]{Bosch-neron}, the Weil restrictions are representable by
group schemes affine and of finite presentation over $T$. Then $N_{m,n}$ is a finitely
presented closed subgroup of $\pi_{n*}G_{(nx)_T}$ (the identity section of $\pi_{m*}G_{(mx)_T}$
is a finitely presented closed immersion). Hence $N_{m,n}$ is representable by a scheme affine 
and of finite presentation over $T$. 
A $T'$-point of $N_{m,n}$ is an element  
$g \in G((nx)_{T'})$ such that $i_{m,n}^*(g) = 1$. 
From this description, it is evident that composition defines a canonical simply transitive
right $N_{m,n}$-action on $F$. 
We show that $F$ is an $N_{m,n}$-torsor. 

From the proof of Proposition~\ref{nx}, we can choose an fppf covering $(T_i \to T)$
such that $i_n^*\cP$ admits trivializations $\gamma_i : (nx)_{T_i} \xt G \simeq 
i_n^*\cP_{T_i}$. Since $(mx)_T,(nx)_T$ are finite over $T$, we 
may additionally assume that $(mx)_{T_i} \into (nx)_{T_i}$ is a nilpotent thickening of
affine schemes. Since $G$ is smooth, by the infinitesimal lifting property 
for smooth morphisms \cite[2.2, Proposition 6]{Bosch-neron}, we have a surjection 
\[ G( (nx)_{T_i} ) \onto G( (mx)_{T_i}). \]
Therefore there exists automorphisms $\phi_i$ of $(nx)_{T_i} \xt G$ 
such that \[ i_{m,n}^*(\phi_i) = i_{m,n}^*(\gamma_i^{-1})\circ \psi_{T_i}.\]
We have shown that $\gamma_i \circ \phi_i \in F(U_i) \ne \emptyset$. It follows that 
$F$ is an $N_{m,n}$-torsor, and $F \to T$ has the desired properties by descent theory. 
\end{proof}

Once again, our plan for proving Theorem~\ref{compact-open-level}
is to reduce to considering $\GL_r$-bundles. The following lemma makes this possible. 

\begin{lem} \label{level-H-G} 
Let $H \into G$ be a closed subgroup of $G$. There is a finitely presented 
closed immersion $\Bun_H^{(nx)} \into \Bun_G^{(nx)} \xt_{\Bun_G} \Bun_H$. 
\end{lem} 

\begin{proof}
The change of group morphism
$\twist{(-)}G : BH \to BG$ of Lemma~\ref{BH-BG} induces a $2$-commutative square
\[ \xymatrix{ \Bun_H^{(nx)} \ar[r] \ar[d] & \Bun_G^{(nx)} \ar[d] \\ 
\Bun_H \ar[r] & \Bun_G } \]
The $2$-fibered product $\Bun_G^{(nx)} \xt_{\Bun_G} \Bun_H$ is isomorphic to 
the stack $\cF : (\Sch_{/S})^\op \to \Gpd$ defined by
\[ 
\cF(T) = \bigl\{ (\cP,\phi) \mid \cP \in \Bun_H(T),\, \phi : (nx)_T \xt G \to 
\twist{(i_n^*\cP)}G \bigr\}, 
\]
where a morphism $(\cP,\phi) \to (\cP,\phi')$ is an $H$-equivariant morphism $f: \cP \to \cP'$
satisfying $\twist{(i_n^*f)}G \circ \phi = \phi'$. 
We show that the morphism $\Bun_H^{(nx)} \to \cF$
induced by the $2$-commutative square is a finitely presented closed immersion.
Fix an $S$-scheme $T$ and take $(\cP,\phi) \in \cF(T)$. 
We have a $2$-commutative square 
\[ \xymatrix{ (nx)_T \ar[d]_{i_n^*\cP} \ar@{=}[r] & (nx)_T \ar[d]^{(nx)_T \xt G} \\
BH \ar[r] & BG } \]
via $\phi$. From the proof of Lemma~\ref{BH-BG}, we know that
$(nx)_T \xt_{BG} BH \simeq (nx)_T \xt (H \bs G)$. This implies that
\[ (nx)_T \xt_{\ds  (i_n^*\cP,\phi), (nx)_T \xt_{BG} BH, ((nx)_T \xt H,\id) } (nx)_T \]
is representable by a finitely presented closed subscheme $T_0 \subset (nx)_T$. 
From the definition of the $2$-fibered product, we have that
$\Hom(T', T_0)$ consists of the morphisms $u : T' \to (nx)_T$ such that there exists
a $\psi : T' \xt H \to u^*i_n^*\cP$ with $\twist{(\psi)}G = u^*(\phi)$. 
If such a $\psi$ exists, then it must be unique by definition of the $2$-fibered product
and the fact that $T_0$ is a scheme. From this description, we observe that
\[ \Bun_H^{(nx)} \xt_\cF T \simeq \pi_{n*} T_0. \]
Then \cite[7.6, Propositions 2, 5]{Bosch-neron} imply that $\pi_{n*}T_0$
is a finitely presented closed subscheme of $\pi_{n*}(nx)_T \simeq T$.
\end{proof}

Lemma~\ref{GL_r-bundle} implies that 
$\Bun_{\GL_r}^{(nx)} \simeq \Bun_r^{(nx)}$, where $\Bun_r^{(nx)}(T)$ is the
groupoid of pairs $(\cE,\phi)$ for $\cE$ a locally free $\cO_{X_T}$-module of rank $r$
and an isomorphism \[\phi : \cO_{(nx)_T}^r \simeq i_n^*\cE\] of $\cO_{(nx)_T}$-modules.
We show that the projections $\Bun_r^{(nx)} \to \Bun_r $
give smooth presentations of the open substacks $\cU_n^\Phi \into \Bun_r$ defined in 
\S\ref{section:Bun_G-presentation}. 

\begin{thm}\label{level-scheme} 
Suppose $p:X \to S$ is a flat, strongly projective morphism with geometrically integral fibers 
over a quasi-compact base scheme $S$. For any integer $n$ and polynomial $\Phi \in \bQ[\lambda]$,
there exists an integer $N_0$
such that for all $N \ge N_0$, the $2$-fibered product 
$\Bun_r^{(Nx)} \xt_{\Bun_r} \cU_n^\Phi$ is representable by a finitely presented,
quasi-projective $S$-scheme.
\end{thm}

Given Theorem~\ref{level-scheme}, let us deduce Theorem~\ref{compact-open-level}.

\begin{proof}[Proof of Theorem~\ref{compact-open-level}]
Embed $G$ into some $\GL_r$.
Take an open covering $(S_i \subset S)_{i\in I}$
by affine subschemes such that the restrictions $X_{S_i} \to S_i$ are strongly projective. 
Let $J' = I \xt \bZ \xt \bQ[\lambda]$, and for $j = (i,n,\Phi) \in J'$, denote $\cU_j = 
\cU_n^\Phi \xt_S S_i$. 
Theorem~\ref{level-scheme} implies that for any $j \in J'$, there exists
an integer $N_{j,0}$ such that for all $N \ge N_{j,0}$, the $2$-fibered product 
$\Bun_r^{(Nx)} \xt_{\Bun_r} \cU_j$
is representable by a scheme. Let $\wtilde \cU_j = 
\Bun_G \xt_{\Bun_r} \cU_j$. 
By quasi-compactness of $\cV$, there exists a finite subset $J \subset J'$ such that 
\[ \cV \subset \bigcup_{j \in J} \wtilde \cU_j  \]
is an open immersion \cite[Lemmas \href{http://math.columbia.edu/algebraic_geometry/stacks-git/locate.php?tag=05UQ}{05UQ}, \href{http://math.columbia.edu/algebraic_geometry/stacks-git/locate.php?tag=05UR}{05UR}]{stacks-project}.
Letting $N_0$ equal the maximum of the $N_{j,0}$ over all $j \in J$, we have that
for all $N \ge N_0,\, j \in J$,
the $2$-fibered product $\Bun_r^{(Nx)} \xt_{\Bun_r} \cU_j$ is representable by an 
scheme. Corollary~\ref{BunG-cog} implies that $\Bun_G \to \Bun_r$ is schematic, 
and Lemma~\ref{level-H-G} shows that 
$\Bun_G^{(Nx)} \to \Bun_r^{(Nx)} \xt_{\Bun_r} \Bun_G$ is a closed 
immersion. Therefore by base change we deduce that 
\[ \Bun_G^{(Nx)} \xt_{\Bun_G} \wtilde \cU_j \] 
is representable by a scheme. 
Now it follows from \cite[Lemma \href{http://math.columbia.edu/algebraic_geometry/stacks-git/locate.php?tag=05WF}{05WF}]{stacks-project} that $\Bun_G^{(Nx)} \xt_{\Bun_G} \cV$ 
is representable by an $S$-scheme. 

From Theorem~\ref{BunG-alg} we know that $\cV$ is a quasi-compact algebraic stack
locally of finite presentation over $S$. 
Proposition~\ref{nx} implies that
$\Bun_G^{(Nx)} \to \Bun_G$ is schematic and finitely presented, so we deduce that
$\Bun_G^{(Nx)}\xt_{\Bun_G} \cV$ is represented by a quasi-compact scheme locally of 
finite presentation over $S$.
\end{proof}

The rest of this section is devoted to proving Theorem~\ref{level-scheme}.

\begin{lem}\label{section-res-field} 
Suppose $S$ is the spectrum of a field $k'$, and $X$ is integral and proper. 
For a locally free $\cO_X$-module $\cF$ of finite rank, there exists an integer $N_0$ such that
for all integers $N \ge N_0$, the unit morphism 
\[ \eta_N: \Gamma(X,\cF) \to \Gamma(X, i_{N*} i_N^* \cF) \simeq 
\Gamma(X, \cF \ot_{\cO_X} \cO_X/\cI_x^N) \]
is injective. \end{lem}
\begin{proof} Let $N$ be an arbitrary positive integer.
Since $\cF$ is $X$-flat, we have a short exact sequence \[ 0 \to \cF \ot_{\cO_X} \cI_x^N 
\to \cF \to \cF \ot_{\cO_X} \cO_X/\cI^N_x \to 0.\] Thus the kernel of $\eta_N$ is 
$\Gamma(X,\cF \ot \cI_x^N)$. Since $X$ is proper,  
$\Gamma(X,\cF \ot \cI_x^N)$ is a finite dimensional $k'$-vector space for each $N$,
and we have a descending chain 
\[ \cdots \supset \Gamma(X,\cF \ot \cI_x^N) \supset \Gamma(X, \cF \ot \cI_x^{N+1}) \supset \cdots
\]
and this chain must stabilize. Suppose the chain does not stabilize to $0$. 
Then there is a nonzero global section $f \in \Gamma(X,\cF)$ lying in $\Gamma(X,\cF \ot \cI_x^N)$
for all $N$. Thus the image of $f$ in \[ \Gamma(X,\cF) \to \cF_x \to \what{\cF}_x \]
is $0$. Since $X$ is assume integral, $\Gamma(X,\cF) \into \cF_x$ is injective.
By Krull's intersection theorem \cite[11.D, Corollary 3]{Matsumura}, the morphism 
$\cF_x \into \what\cF_x$ is injective. Therefore $f = 0$, a contradiction.
We conclude that there exists $N_0$ such that for all $N \ge N_0$ for some $N_0$, 
the kernel $\ker \eta_N = \Gamma(X,\cF \ot \cI^N_x) = 0$.
\end{proof}

\begin{lem}\label{section-res} Suppose $p : X \to S$ is a flat, strongly projective morphism
with geometrically integral fibers. 
For a quasi-compact $S$-scheme $T$, let $\cE \in \cU_n(T)$. Then there exists an integer 
$N_0$ such that for all integers $N \ge N_0$, the dual of the unit morphism 
\begin{equation}\label{eqn-dual-unit} 
\bigl(p_{T*} i_{N*}i_N^*(\cE(n)) \bigr)^{\!\vee} \to \bigl(p_{T*}(\cE(n)) \bigr)^{\!\vee}
\end{equation}
is a surjective morphism of locally free $\cO_T$-modules. 
\end{lem}
\begin{proof}
Let $\cF = \cE(n)$. From Proposition~\ref{U_n-functor}, we know that $p_{T*}\cF$ is
flat, and $\cF$ is cohomologically flat over $T$ in all degrees. Let $N$ be an 
arbitrary positive integer. By the projection formula,
$i_{N*}i_N^* \cF \simeq \cF \ot_{\cO_{X_T}} i_{N*}\cO_{(Nx)_T}$. Since $\pi_N : (Nx)_T 
\to T$ is finite and locally free, $i_{N*} \cO_{(Nx)_T}$ is $T$-flat. We deduce that
$i_{N*}i_N^*\cF$ is $T$-flat because $\cF$ is locally free on $X_T$. Additionally, 
$p_{T*} i_{N*} i_N^* \cF = \pi_{N*} (i_N^*\cF)$ is locally free on $T$ because 
$i_N^*\cF$ and $\pi_N$ are locally free. By the Leray spectral sequence, 
we have quasi-isomorphisms 
\[ Rp_{T*}i_{N*}i_N^*\cF \simeq R\pi_{N*}i_N^*\cF \simeq \pi_{N*}i_N^*\cF \]
in the derived category $D(T)$ of $\cO_T$-modules, since $\pi_N$ is affine. 
Now Lemma~\ref{cohom-flat} implies 
$i_{N*}i_N^* \cF$ is cohomologically flat over $T$ in all degrees. 
Remark~\ref{direct-image-free} implies $p_{T*}\cF$ is locally free, so 
$p_{T*} \cF \to p_{T*} i_{N*} i_N^* \cF$ is a morphism of locally free $\cO_T$-modules. 
Take a point $t \in T$. By Lemma~\ref{section-res-field}, there exists an integer
$N_t$ such that for all $N \ge N_t$, the unit morphism $\Gamma(X_t, \cF_t) \to 
\Gamma(X_t, i_{N*}i_N^*\cF_t)$ of finite dimensional $\kappa(t)$-vector spaces is injective. 
By cohomological flatness, this implies 
\[ (p_{T*}\cF) \ot_{\cO_T} \kappa(t) \to (p_{T*}i_{N*}i_N^*\cF) \ot_{\cO_T} \kappa(t) \]
is injective (observe that $i_{N*}i_N^*$ commutes with base change).
Since taking duals commutes with base change for locally free modules, we have that
$(p_{T*} i_{N*} i_N^* \cF)^\vee \ot \kappa(t) \to (p_{T*}\cF)^\vee \ot \kappa(t)$ is
surjective. By Nakayama's lemma, there exists an open subscheme $U_t \subset T$ containing $t$
on which the restriction $(p_{T*} i_{N*} i_N^* \cF)^\vee |_{U_t} \to (p_{T*}\cF)^\vee |_{U_t}$ 
is surjective. By quasi-compactness, we can cover $T$ by finitely many $U_t$. Taking
$N_0$ to be the maximum of the corresponding $N_t$ proves the claim.
\end{proof}

\begin{lem}\label{section-res-family} 
Suppose $p:X \to S$ is a flat, strongly projective morphism with geometrically integral fibers 
over a quasi-compact base scheme $S$. 
For any integer $n$ and polynomial $\Phi \in \bQ[\lambda]$, there exists an integer 
$N_0$ such that for all $N \ge N_0$, the morphism \eqref{eqn-dual-unit} is a surjective morphism
of locally free $\cO_T$-modules for any $S$-scheme $T$ and $\cE \in \cU_n^\Phi(T)$. 
\end{lem}
\begin{proof}
From Theorem~\ref{BunG-alg} we get a quasi-compact scheme $Y_n^\Phi$ and a schematic,
smooth, surjective morphism $Y_n^\Phi \to \cU_n^\Phi$. By the $2$-Yoneda lemma, 
the previous morphism corresponds to a locally free sheaf $\cE_0 \in \cU_n^\Phi(Y_n^\Phi)$.
By Lemma~\ref{section-res}, there exists $N_0$ such that for all $N \ge N_0$, 
the morphism \eqref{eqn-dual-unit} is surjective for $\cE_0$. 
Now for any $S$-scheme $T$ and $\cE \in \cU_n^\Phi(T)$, let 
\[ \xymatrix{ T' \ar[r] \ar[d] & T \ar[d]^{\cE} \\ Y_n^\Phi \ar[r]^{\cE_0} & \cU_n^\Phi } \]
be $2$-Cartesian, where $T'$ is a scheme and $T' \to T$ is smooth and surjective (and hence
faithfully flat). Thus there is an isomorphism $(\cE_0)_{T'} \simeq \cE_{T'}$. 
From the proof of Lemma~\ref{section-res}, we know that 
$i_{N*}i_N^*\cE_0(n)$ and $\cE_0(n)$ are cohomologically flat over $Y_n^\Phi$, and
the analogous assertion is true for $\cE$ over $T$. Therefore pulling back along $T' \to Y_n^\Phi$
implies that \eqref{eqn-dual-unit} is surjective for $\cE_{T'} \in \cU_n^\Phi(T')$. 
Since $T' \to T$ is faithfully flat, we deduce that \eqref{eqn-dual-unit} is surjective
for $\cE \in \cU_n^\Phi(T)$.
\end{proof}

\begin{proof}[Proof of Theorem~\ref{level-scheme}]  
Let $N_0$ be an integer satisfying the assertions of Lemma~\ref{section-res-family}.
Fix an integer $N \ge N_0$, and consider the coherent sheaf 
$\cM = \bigl(\pi_{N*}i_N^*(\cO_X^r(n)) \bigr)^{\!\vee}$, which is locally free of finite rank
since $i_N^*(\cO_X^r(n))$ and $\pi_N$ are. 
By \cite[Proposition 9.7.7, 9.8.4]{EGA1-2ed}, the Grassmannian functor $\Grass(\cM,\Phi(n))$
is representable by a strongly projective $S$-scheme. 
We have a morphism 
\[ \Grass(\cM,\Phi(n)) \to B\GL_{\Phi(n)}\xt S \] sending a $T$-point
$\cM_T \onto \cF$ to $\cF^\vee$. 
From Lemmas~\ref{Z/G=X} and \ref{Y_n-U_n}, we have an isomorphism $\cU_n^\Phi \simeq
[Y_n^\Phi/\GL_{\Phi(n)}]$ which sends $\cE \in \cU_n^\Phi(T)$ to
$\Isom_T \bigl(\cO_T^{\Phi(n)},p_{T*}(\cE(n))\bigr) \to Y_n^\Phi$, where the $\GL_{\Phi(n)}$-bundle
corresponds to the locally free $\cO_T$-module $p_{T*}(\cE(n))$. 
Therefore the change of space morphism 
$[Y_n^\Phi/\GL_{\Phi(n)}] \to [S/\GL_{\Phi(n)}]$ is isomorphic to the morphism 
$\cU_n^\Phi \to B\GL_{\Phi(n)} \xt S$ sending $\cE \mapsto p_{T*}(\cE(n))$ by 
Remark~\ref{BGxZ}. By Lemmas~\ref{cos-representable} 
and \ref{Y_n-equivariant-line-bundle}, 
we deduce that the morphism $\cU_n^\Phi \to B\GL_{\Phi(n)} \xt S$ is schematic, 
quasi-projective, and finitely presented. We have a Cartesian square
\[ \xymatrix{ F_0 \ar[r] \ar[d] & \Grass(\cM,\Phi(n)) \ar[d] \\
\cU_n^\Phi \ar[r] & B\GL_{\Phi(n)} \xt S } \]
where $F_0$ is a functor $(\Sch_{/S})^\op \to \Set$ sending 
\[ (T \to S) \mapsto \bigl\{ (\cE, q) \mid \cE \in \cU_n^\Phi(T), q : \cM
\onto (p_{T*}(\cE(n)))^\vee \bigr\}/{\sim} \]
where $(\cE,q) \sim (\cE',q')$ if there exists (a necessarily unique) isomorphism
$f : \cE \simeq \cE'$ such that $(p_{T*}(f_n))^\vee \circ q' = q$. Note that
$F_0$ is representable by a scheme finitely presented and quasi-projective over 
$\Grass(\cM,\Phi(n))$. By \cite[Proposition 5.3.4(ii)]{EGA2},
we find that $F_0$ is finitely presented and quasi-projective over $S$.

For an $S$-scheme $T$, let $(\cE,\phi),(\cE',\phi') \in \Bun_r^{(Nx)}(T)$ such that 
$\cE,\cE' \in \cU_n^\Phi(T)$. Suppose we have a morphism $f: (\cE,\phi) \to (\cE',\phi')$
in $\Bun_r^{(Nx)}$. Then Lemma~\ref{section-res} gives us a commutative diagram
\[ \xymatrix { \cM_T \ar[rr]^-{(\pi_{N*}(\phi_n'^{-1}))^\vee}_-\sim \ar@{=}[d] & & 
\bigl(p_{T*}i_{N*} i_N^*(\cE'(n)) \bigr)^{\!\vee} \ar@{->>}[r] 
\ar[d]^{(p_{T*}i_{N*}i_N^*(f_n) )^\vee} 
& \bigl(p_{T*}(\cE'(n))\bigr)^{\!\vee} \ar[d]^{(p_{T*}(f_n))^\vee}  \\ 
\cM_T \ar[rr]^-{(\pi_{N*}(\phi_n^{-1}))^\vee}_-\sim & & 
\bigl(p_{T*}i_{N*}i_N^*(\cE(n)) \bigr)^{\!\vee} \ar@{->>}[r] & \bigl(p_{T*}(\cE(n))\bigr)^{\!\vee} 
}\]
where the horizontal arrows $q',q$ are surjections. Observe that $f$ gives an equivalence
$(\cE,q) \sim (\cE',q')$ in $F_0(T)$, and consequently $f$ must be unique. 
Therefore $\Bun_r^{(Nx)} \xt_{\Bun_r} \cU_n^\Phi$ is
isomorphic to the functor $F: (\Sch_{/S})^\op \to \Set$ sending 
\[ (T \to S) \mapsto \bigl\{ (\cE,\phi) \mid \cE \in \cU_n^\Phi(T), \phi : \cO^r_{(Nx)_T} 
\simeq i_N^*\cE \bigr\}/{\sim}. \]
By the previous discussion, we have a morphism $F \to F_0$ by sending 
$(\cE,\phi) \mapsto (\cE,q)$. We show that this morphism is a finitely presented,
locally closed immersion, which will then imply that $F$ is representable by
a finitely presented, quasi-projective $S$-scheme.

The claim is Zariski local on $S$, so we reduce to the case where $S$ is noetherian
by Remark~\ref{noetherian-base}. Let $T$ be an $S$-scheme and $(\cE,q) \in F_0(T)$. 
Then $\cR = \pi_{N*}\cO_{(Nx)_T}$ is a coherent $\cO_T$-algebra on $T$, and 
since $\pi_N$ is affine, \cite[Proposition 1.4.3]{EGA2} implies that
\[ \pi_{N*} : \QCoh((Nx)_T) \to \QCoh(T) \] 
is a faithful embedding to the subcategory
of quasi-coherent $\cO_T$-modules with structures of $\cR$-modules, and
morphisms are those of $\cR$-modules. Since $\cE(n)$ is relatively generated by
global sections, the counit $p_{T}^*p_{T*}(\cE(n)) \to \cE(n)$ is surjective. 
Applying $i_N^*$ gives a surjection $\pi_N^* p_{T*}(\cE(n)) \to i_N^*(\cE(n))$. 
As $\pi_{N*}$ is exact on quasi-coherent sheaves, we have a surjection 
\[ \id_\cR \oot \eta :
\cR \ot_{\cO_T} p_{T*}(\cE(n)) \simeq \pi_{N*} \pi_N^* p_{T*} (\cE(n)) \onto 
\pi_{N*} i_N^* (\cE(n)), 
\]
and this is the morphism of $\cR$-modules induced by the unit
$\eta : p_{T*}(\cE(n)) \to p_{T*}i_{N*}i_N^*(\cE(n))$. Now the dual morphism 
$q^\vee : p_{T*}(\cE(n)) \to \cM_T^\vee$ also induces a morphism of $\cR$-modules
\[ 
\id_\cR \oot q^\vee :
\cR \ot_{\cO_T} p_{T*}(\cE(n)) \to \cM_T^\vee \simeq \pi_{N*}i_N^*(\cO_X^r(n)). 
\]
If $(\cE,q)$ is the image of some $(\cE,\phi) \in F(T)$, then the diagram 
\begin{equation} \label{eqn:q-phi} 
\xymatrix{ & \cR \ot p_{T*}(\cE(n)) \ar[ld]_{\id_\cR \oot q^\vee} 
\ar@{->>}[rd]^{\id_\cR \oot \eta} \\ 
\pi_{N*}i_N^*(\cO_X^r(n)) \ar[rr]_\sim^{\pi_{N*}\phi_n} && \pi_{N*} i_N^*(\cE(n)) } 
\end{equation}
is commutative. In this case, $\id_\cR \oot q^\vee$ is surjective. 
Since $\pi_{N*}$ is faithful, \eqref{eqn:q-phi} implies that 
if such a $\phi$ exists, then it is unique. Conversely, if $\id_\cR \oot q^\vee$ is surjective
and a factorization as in \eqref{eqn:q-phi} exists, then such a $\phi$ does exist. 
Indeed, any morphism
\[ \pi_{N*} i_N^*(\cO^r_X(n)) \to \pi_{N*} i_N^*(\cE(n)) \]
making the triangle commute is necessarily a surjective morphism of $\cR$-modules. 
Therefore it is in the image
of $\pi_{N*}$. Tensoring by $i_N^*(\cO(-n))$ gives the desired $\phi$,
which is an isomorphism since it is a surjection of locally free $\cO_{(Nx)_T}$-modules 
of rank $r$. Define the open subscheme $U \subset T$ to be the
complement of the support of $\coker( \id_\cR \oot q^\vee )$. 
By cohomological flatness (see Lemma~\ref{section-res}), 
the morphisms $\eta, q^\vee$ commute with pullback along a morphism $T' \to T$. 
We also have $\cR_{T'} \simeq \pi_{N*} \cO_{(Nx)_{T'}}$ since $\pi_N$ is affine. 
Therefore $\id_\cR \oot \eta,\, \id_\cR \oot q^\vee$ commute with pullback along $T' \to T$.
By Nakayama's lemma and faithful flatness of field extensions, we deduce that 
$\id_{\cR_{T'}} \oot q_{T'}^\vee$ is surjective if and only if $T' \to T$ lands in $U$. 
Let 
\[ \cK = \ker( \id_\cR \oot q^\vee |_U), \]
which is a coherent $\cO_U$-module. By \cite[Lemme 9.7.9.1]{EGA1-2ed}, there exists
a closed subscheme $V \subset U$ with the universal property that a morphism
$u : T' \to U$ factors through $V$ if and only if $u^*(\id_\cR \oot \eta)(\cK_{T'}) = 0$.
Since $\id_\cR \oot q^\vee|_U$ is surjective and $\pi_{N*}i_N^*(\cO^r_X(n))|_U$ is 
$U$-flat, we have $\cK_{T'} \simeq \ker( u^*(\id_\cR \oot q^\vee ) )$. 
Therefore $u$ factors through $V$ if and only if the diagram \eqref{eqn:q-phi} admits
a factorization on $T'$. We conclude that $V$ represents the fibered product 
$F \xt_{F_0} T$. Taking $T = F_0$, which is noetherian, we see that
$F \to F_0$ is a finitely presented immersion.
\end{proof}

\section{Smoothness}\label{section:smoothness}

We have already seen that $\Bun_G$ is locally of finite presentation. We now show
that $\Bun_G$ is smooth over $S$ when $G$ is smooth and $X \to S$ is a relative curve. 

\begin{prop} \label{smoothness} 
Suppose $G$ is smooth over $k$ and 
$p : X \to S$ is a flat, finitely presented, projective morphism of $k$-schemes with 
fibers of dimension $1$. Then $\Bun_G$ is smooth over $S$. 
\end{prop}

We remind the reader that, following \cite[D\'efinition 4.14]{LMB}, an algebraic $S$-stack 
$\cX$ is smooth over $S$ if there exists a scheme $U$ smooth over $S$
and a smooth surjective morphism $U \to \cX$. 

\subsection{Infinitesimal lifting criterion} 
We first provide an infinitesimal lifting criterion for smoothness of an algebraic $S$-stack.
The next lemma in fact follows from the proofs of \cite[2.2, Proposition 6]{Bosch-neron}
and \cite[Proposition 4.15]{LMB}, but we reproduce the proof to emphasize that we can impose
a noetherian assumption. 

\begin{lem} \label{stack-smooth-lift} 
Let $S$ be a locally noetherian scheme and $\cX$ an algebraic stack locally of finite type 
over $S$ with a schematic diagonal. 
Suppose that for any affine noetherian scheme $T$ and a closed
subscheme $T_0$ defined by a square zero ideal, given a $2$-commutative square
\[ \xymatrix{ T_0 \ar[r] \ar[d] & \cX \ar[d] \\ T \ar@{-->}[ru] \ar[r] & S } \]
of solid arrows, there exists a lift $T \to \cX$ making the diagram $2$-commutative. 
Then $\cX \to S$ is smooth. 
\end{lem}

\begin{proof}
Let $U$ be a scheme locally of finite type over $S$ with a smooth surjective morphism 
to $\cX$. We want to show $U \to S$ is smooth. This is local on $S$, so we may
assume $S = \spec R$ is an affine noetherian scheme. Suppose we have a morphism 
$T_0 \to U$ over $S$. Then the hypothesis allows us to find a lift $T \to \cX$. 
We have a $2$-commutative diagram 
\[ \xymatrix{ T_0 \ar[r] \ar[rd] & U \xt_\cX T \ar[r] \ar[d] & U \ar[d] \\ 
& T \ar[r] & \cX } \]
where $U \xt_\cX T \to T$ is isomorphic to a smooth morphism of scheme. 
By \cite[2.2, Proposition 6]{Bosch-neron}, there exists a section $T \to U \xt_\cX T$. 
This shows that 
\begin{equation} \label{eqn:smooth-lift}
\Hom_S(T,U) \onto \Hom_S(T_0, U) 
\end{equation}
is surjective. Note that \eqref{eqn:smooth-lift} still holds if we replace $U$ by
an open subscheme, since $T_0$ and $T$ have the same underlying topological space. 
Since smoothness is local on the source, we can assume $U = \spec B$ is affine. 
Then $U \to S$ of finite type implies that $U$ is a closed subscheme of $\bA_S^n$. 
Let $A = R[t_1,\dots,t_n]$ and $I \subset A$ the ideal corresponding to $U \into \bA_S^n$.
By \cite[II, Th\'eor\`eme 4.10]{SGA1}, it suffices to show that the canonical sequence
\[
0 \to \cI/\cI^2 \to \Omega_{\bA_S^n/S} \ot \cO_U \to \Omega_{U/S} \to 0 
\]
is locally split exact (we know {\it a priori} that the sequence is right exact).
Note that $A$ is noetherian, and $\cI/\cI^2$ is the coherent sheaf corresponding
to $I/I^2$. Since $A/I^2$ is noetherian, by \eqref{eqn:smooth-lift} the $R$-morphism 
\[ \id : A/I \to A/I = (A/I^2)/(I/I^2) \]
lifts to a morphism $\varphi : A/I \to A/I^2$ of $R$-algebras. This implies that 
the short exact sequence
\[ 0 \to I/I^2 \oarrow\iota A/I^2 \oarrow\pi A/I \to 0 \]
splits since $\pi \circ \varphi = \id_{A/I}$. Let 
$\delta = \id_{A/I^2} - \varphi \circ \pi : A/I^2 \to I/I^2$ be a left inverse to $\iota$.
Since $\delta(a)\delta(b) = 0$ for all $a,b \in A/I^2$, we have 
\[ \delta(ab) = ab - (\varphi \circ \pi)(ab) + (a-(\varphi \circ \pi)(a))(b - 
(\varphi \circ \pi)(b)) = a\delta(b) + b \delta(a) \]
using the fact that $\varphi \circ \pi$ is a morphism of $R$-algebras.
Thus $\delta$ gives an $R$-derivation $A \to I/I^2$, which corresponds to a morphism 
$\Omega_{A/R} \to I/I^2$ of $A$-modules. Since $\delta \circ \iota = \id_{I/I^2}$,
this morphism defines a left inverse $\Omega_{A/R} \ot_A B \to I/I^2$, which
shows that 
\[ 0 \to I/I^2 \to \Omega_{A/R} \ot_A B \to \Omega_{B/R} \to 0 \]
is split exact. We conclude that $U \to S$ is smooth. 
\end{proof}

\subsection{Lifting gerbes} 
Our aim is to prove Proposition~\ref{smoothness} by showing that $\Bun_G$
satisfies the lifting criterion of Lemma~\ref{stack-smooth-lift}. In other words, 
we want to lift certain torsors. 
To discuss when torsors lift, we will need to use the categorical language of a
gerbe over a Picard stack. We refer the reader to \cite[3]{Donagi-DG} or \cite{Giraud} 
for the relevant definitions. 
\smallskip

Let $\cC$ be a subcanonical site and $X$ a terminal object of $\cC$. 
Suppose we have a short exact sequence of sheaves of groups
\[ 1 \to \cA \to \cG \oarrow\pi \cG_0 \to 1 \]
on $\cC$, where $\cA$ is abelian. Then $\cG$ acts on $\cA$ via conjugation, 
which induces a left action of $\cG_0$ on $\cA$, since $\cA$ is abelian. Fix a right 
$\cG_0$-torsor $\cP$ over $X$. Then the twisted sheaf ${_\cP\cA}$ is
still a sheaf of abelian groups because the $\cG_0$-action is induced by conjugation.
We define a stack $\cQ_\cP : \cC \to \Gpd$ by letting
\[ \cQ_\cP(U) = \bigl\{ (\wtilde \cP, \phi) \mid \wtilde \cP \in \Tors(\cG)(U),\, \phi : \cP|_U \to 
\twist{\wtilde \cP}{\cG_0} \bigr\} \]
where $\phi$ is $\cG_0$-equivariant, and a morphism $(\wtilde \cP,\phi) \to ({\wtilde \cP}',\phi')$
in the groupoid is a $\cG$-equivariant morphism $f : \wtilde \cP \to {\wtilde \cP}'$ such that
$\twist{f}{\cG_0} \circ \phi = \phi'$. 

\begin{lem}\label{lift-gerbe}
The stack $\cQ_\cP$ of liftings of $\cP$ to 
$\cG$ is a gerbe over the Picard stack $\Tors(\twist\cP\cA)$.
\end{lem}

\begin{proof}
We first describe the action of $\Tors(\twist\cP\cA)(U)$ on $\cQ_\cP(U)$ for 
$U \in \cC$. Let $\cT,\, (\wtilde \cP,\phi)$ be
objects in the aforementioned categories, respectively. We can pick a covering $(U_i \to U)$ 
trivializing $\cP|_U, \cT$, and $\wtilde \cP$. 
This in particular implies that $\cT|_{U_i} \simeq \cA|_{U_i}$. 
Denote $U_{ij} = U_i \xt_U U_j$ and 
$U_{ijk} = U_i \xt_U U_j \xt_U U_k$ for indices $i,j,k$. 
Let 
\[ (\cG_0|_{U_{ij}}, g_{ij}) \]
be a descent datum for $\cP|_U$, where
$g_{ij} \in \cG_0(U_{ij})$. Then $\twist\cP\cA|_U$ has a descent datum 
$(\cA|_{U_{ij}}, \gamma_{ij})$
with $\gamma_{ij}(a) = g_{ij} a g_{ij}^{-1}$. Now since $\cT$ is a $\twist\cP\cA|_U$-torsor,
the transition morphisms in a descent datum for $\cT$ are $\twist\cP\cA|_U$-equivariant. 
Thus $\cT$ has a descent datum $(\cA|_{U_{ij}}, \psi_{ij})$
where 
\[ \psi_{ij}(a) = \gamma_{ij}(a_{ij} a) = g_{ij} a_{ij} a g_{ij}^{-1} \]
for $a_{ij} \in \cA(U_{ij})$ under the group operation of $\cG$ (we omit restriction symbols). 
Now let $(\cG|_{U_{ij}}, h_{ij})$ be a descent datum for $\wtilde \cP$, where 
$h_{ij} \in \cG(U_{ij})$. The cocycle condition says that $h_{ij} h_{jk} = h_{ik} 
\in \cG(U_{ijk})$. 
Now $\phi$ corresponds to a morphism of descent data 
\[ (\phi_i) : (\cG_0|_{U_i}, g_{ij}) \to (\cG_0|_{U_i}, \pi(h_{ij}) ). \]
By possibly refining our cover further,
we may assume that $\phi_i = \pi(c_i)$ for $c_i \in \cG(U_i)$. 
Set $\wtilde g_{ij} = c_i^{-1} h_{ij} c_j \in \cG(U_{ij})$. Now 
$\wtilde g_{ij} \wtilde g_{jk} =  \wtilde g_{ik}$
and $\pi(\wtilde g_{ij}) = g_{ij}$. Therefore the cocycle condition $\psi_{ij} \circ \psi_{jk}
= \psi_{ik}$ is equivalent to
\begin{equation} \label{eqn:gerbe}
\wtilde g_{ij} a_{ij} \wtilde g_{jk} a_{jk} \wtilde g_{jk}^{-1} \wtilde g_{ij}^{-1} 
= \wtilde g_{ik} a_{ik} \wtilde g_{ik}^{-1}
\iff h_{ij} c_j a_{ij} c_j^{-1} h_{jk} c_k a_{jk} = h_{ik} c_k a_{ik}. 
\end{equation}
Therefore $(\cG|_{U_i}, h_{ij} c_j a_{ij} c_j^{-1})$
defines a descent datum. Note that $c_j a_{ij} c_j^{-1}$ does not depend on our choice of $c_i$. 
Therefore this descent datum defines a canonical $\cG|_U$-torsor $\wtilde \cP \ot \cT$. 
Since $\pi(h_{ij}a'_{ij}) = \pi(h_{ij})$, we see that $\phi$ defines a morphism 
$\cP|_U \to \twist{(\wtilde \cP \ot \cT)}{\cG_0}$. 
\smallskip

Suppose we have a morphism of descent data 
\[ (f_i) : (\cG|_{U_i}, h_{ij}) \to (\cG|_{U_i}, h'_{ij}),\, f_i \in \cG(U_i) \] 
corresponding to a morphism $(\wtilde \cP, \phi) \to 
(\wtilde \cP', \phi')$ in $\cQ_\cP(U)$. 
We refine the cover to assume $\phi_i = \pi(c_i)$ and $\phi'_i = 
\pi(c'_i)$ (notation as in the previous paragraph). Then from the definitions we have
$h'_{ij} f_j = f_i h_{ij}$ and $\pi(f_i c_i) = \pi(c'_i)$. Thus 
\[ h'_{ij} c'_j a_{ij} c_j'^{-1} f_j = h'_{ij} f_j c_j a_{ij} c_j^{-1} = f_i h_{ij} 
c_j a_{ij} c_j^{-1}.\]
Therefore $(f_i)$ defines a morphism $(\wtilde \cP \ot \cT,\phi) \to
(\wtilde \cP' \ot \cT,\phi')$.
\smallskip

Let $(b_i) : (\cA|_{U_i}, \psi_{ij}) \to (\cA|_{U_i}, \psi'_{ij}),\, b_i \in \cA(U_i)$ 
be a morphism of descent data corresponding to a morphism $\cT \to \cT'$ 
in $\Tors(\twist \cP \cA)(U)$. Here $\psi_{ij}, \psi'_{ij}$ are defined with respect to
$a_{ij}, a'_{ij} \in \cA(U_{ij})$, respectively, as described above. 
The compatibility condition $\psi'_{ij}(b_j) = b_i \psi_{ij}(1)$ required for $(b_i)$ to be
a morphism of descent data is equivalent to 
\[ g_{ij} a'_{ij} b_j g_{ij}^{-1} = b_i g_{ij} a_{ij}
g_{ij}^{-1} \iff h_{ij} c_j a'_{ij} c_j^{-1} c_j b_j c_j^{-1} = 
c_i b_i c_i^{-1} h_{ij} c_j a_{ij}c_j^{-1}. \]
Therefore $(c_i b_i c_i^{-1})$ defines a morphism $(\wtilde \cP \ot \cT,\phi) \to
(\wtilde \cP \ot \cT', \phi)$. 
\medskip

We have now defined an action of $\Tors(\twist \cP \cA)$ on $\cQ_\cP$. 
To check that $\cQ_\cP$ is a gerbe over $\Tors(\twist \cP \cA)$, we show that
for a fixed $(\wtilde \cP, \phi) \in \cQ_\cP(U)$, the functor 
\[ \Tors(\twist \cP \cA)(U) \to \cQ_\cP(U) : \cT \mapsto (\wtilde \cP \ot \cT, \phi) \]
is an equivalence of groupoids. We use the notation for descent data from the preceding
paragraphs. First, we show fully faithfulness. 
Let $f_i \in \cG(U_i)$ correspond to a morphism $(\wtilde \cP \ot \cT, \phi) \to 
(\wtilde \cP \ot \cT', \phi)$. The compatibility conditions require that 
$\pi(f_i) = \pi(1)$ and 
\[ h_{ij} c_j a'_{ij} c_j^{-1} f_j = f_i h_{ij} c_j a_{ij} c_j^{-1} \iff 
\psi'_{ij}( c_j^{-1} f_j c_j) = (c_i^{-1} f_i c_i) \psi_{ij}(1). \]
Therefore $b_i = c_i^{-1} f_i c_i \in \cA(U_i)$ corresponds bijectively to a morphism 
$\cT \to \cT'$.
\smallskip

Next, let us show essential surjectivity of the functor. Take $(\wtilde \cP', \phi')
\in \cQ_\cP(U)$. We have $\pi(c_i^{-1} h_{ij} c_j) = \pi(c_i'^{-1} h'_{ij} c'_j) = g_{ij}$. 
Thus there exist unique $a_{ij} \in \cA(U_{ij})$ such that 
\[ c_i'^{-1} h'_{ij} c'_j =  c_i^{-1} h_{ij} c_j a_{ij}. \] 
Since the $h'_{ij}$ satisfy the cocycle condition, \eqref{eqn:gerbe} shows that
the transition morphisms $\psi_{ij}$ associated to $a_{ij}$ satisfy the cocycle condition. 
Therefore $(\cA|_{U_i}, \psi_{ij})$ defines a descent datum, which corresponds to 
a torsor $\cT \in \Tors(\twist \cP \cA)(U)$, and 
$(c'_i c_i^{-1})$ corresponds to a morphism of descent
data for $(\wtilde \cP \ot \cT, \phi) \to (\wtilde \cP', \phi')$. 
Therefore the functor in question is an equivalence, and we conclude that
$\cQ_\cP$ is a gerbe over $\Tors(\twist \cP \cA)$.
\end{proof}

\subsection{Proving smoothness} 

\begin{proof}[Proof of Proposition~\ref{smoothness}]
Smoothness is local on $S$, so we may assume $S$ is noetherian by Remark~\ref{noetherian-base}.
We will prove $\Bun_G \to S$ is smooth by showing that it satisfies the conditions
of Lemma~\ref{stack-smooth-lift}. 
Let $T = \spec A$ be an affine noetherian scheme and $T_0$ a closed subscheme 
defined by a square zero ideal $I \subset A$. Suppose we have 
a $G$-bundle $\cP$ on $X_{T_0}$ corresponding to a morphism $T_0 \to \Bun_G$. 
Since $X_{T_0} \to X_T$ is a homeomorphism of underlying topological spaces, 
\cite[I, Th\'eor\`eme 8.3]{SGA1} implies that base change gives an 
equivalence of categories between small \'etale sites 
\[ (X_T)_{\textrm{\'et}} \to (X_{T_0})_{\textrm{\'et}} : U \mapsto U_0 = U \xt_T T_0. \]
Using this equivalence, we define two sheaves of groups $\cG_0, \cG$ on 
$(X_{T_0})_{\textrm{\'et}}$ by 
\[ \cG_0(U_0) = \Hom_k(U_0,G) \text{ and } \cG(U_0)= \Hom_k(U,G). \]
Since $G$ is smooth, infinitesimal lifting \cite[2, Proposition 6]{Bosch-neron} implies that
$\cG \to \cG_0$ is a surjection of sheaves. Additionally, smoothness of $G$ implies 
that $\cP$ is \'etale locally trivial. Thus $\cP$ is
a $\cG_0$-torsor on $(X_{T_0})_{\textrm{\'et}}$, and a lift of $\cP$ to a $G$-bundle on 
$X_T$ is the same as a $\cG$-torsor inducing $\cP$. This is equivalent to an 
element of the lifting gerbe $\cQ_\cP(X_{T_0})$. 

Consider the short exact sequence of sheaves of groups
\[ 1 \to \cA \to \cG \to \cG_0 \to 1. \]
We give an explicit description of $\cA$. 
Suppose $U = \spec B$ is affine and $U_0 = \spec B/IB$ for an $A$-algebra $B$. 
We will use the notation $k[G] = \Gamma(G,\cO_G)$. 
Let $\vareps : k[G] \to k$ be the morphism of $k$-algebras corresponding to $1 \in G(k)$. Then 
$\cA(U_0)$ consists of the $\varphi \in \Hom_k(k[G],B)$
such that the compositions $k[G] \oarrow\varphi B \to B/IB$ and $k[G] \oarrow{\vareps} k 
\into B/IB$ coincide.
Define the morphism $\delta = \varphi - \vareps : k[G] \to IB$ of
$k$-modules. Observe that for $x,y \in k[G]$, we have $\delta(x)\delta(y)=0$ since $I^2=0$.
Thus 
\[ \varphi(xy) - \varphi(x)\varphi(y) = \delta(xy) - \vareps(x) \delta(y) - 
\vareps(y)\delta(x) =0, \]
which implies that $\varphi$ corresponds bijectively to a derivation
\[ \delta \in \Der_k(k[G], IB) \simeq \Hom_k(\fm_1/\fm_1^2, IB), \]
where $IB$ is a $k[G]$-module via $\vareps$ and $\fm_1 = \ker \vareps$ is the maximal ideal. 
Since $U\to X_T \to T$ is flat, 
$B$ is a $A$-flat. Therefore $IB = I \ot_A B$. Noting that $\fm_1/\fm_1^2 = \fg^\vee$, 
we have 
\[ \cA(U_0) \simeq \Hom_k(\fg^\vee, IB)\simeq 
(\fg \ot_k I) \ot_A B \simeq (\fg \ot_k I) \ot_{A/I} B/IB \]
since $I^2=0$ makes $I$ into an $A/I$-module. Let $\cI$ be the ideal sheaf 
corresponding to $X_{T_0} \into X_T$, which is coherent because we assume $T$ is noetherian. 
We have shown that $\cA$ is the abelian sheaf corresponding to $\fg \ot_k \cI \in \Coh(X_{T_0})$
(see \cite[Lemma \href{http://math.columbia.edu/algebraic_geometry/stacks-git/locate.php?tag=03DT}{03DT}]{stacks-project}). Since $\cG_0$ acts on $\cA$ by conjugation, we see
from the above isomorphisms that $\cG_0$ acts on $\fg \ot_k \cI$ via the 
adjoint representation on $\fg$. Therefore 
\[ \twist \cP \cA \simeq \twist \cP {(\fg)} \ot_{\cO_{X_{T_0}}} \cI, \]
which is a coherent sheaf by quasi-coherent descent \cite[Theorem 4.2.3]{FGA-explained} 
and persistence of finite presentation under fpqc morphisms \cite[Proposition 2.5.2]{EGA4b}.

There is a bijection between equivalence classes of gerbes over $\Tors(\twist \cP \cA)$
and the \'etale cohomology group $H^2( (X_{T_0})_{\textrm{\'et}}, \twist \cP \cA)$ by
\cite[Th\'eor\`eme 3.4.2]{Giraud}. Since $\twist \cP \cA$ is in fact quasi-coherent, 
\cite[VII, Proposition 4.3]{SGA4-2} gives an isomorphism 
\[ H^2( X_{T_0}, \twist \cP \cA ) \simeq H^2( (X_{T_0})_{\textrm{\'et}}, \twist \cP \cA), \]
where the left hand side is cohomology in the Zariski topology. 
By assumption, the fibers of $X \to S$ are of dimension $1$. It follows from 
\cite[Corollaire 4.1.4]{EGA4b} that $X_{T_0} \to T_0$ also has fibers of dimension $1$.
Now since $\twist \cP \cA$ is coherent, \cite[Th\'eoreme 3.6.5]{Grothendieck-Tohoku} and 
\cite[8.5.18]{FGA-explained} imply that $H^2(X_{T_0}, \twist \cP \cA)=0$. 
Therefore by Lemma~\ref{lift-gerbe}, the lifting gerbe $\cQ_\cP$ is neutral,
i.e., $\cQ_\cP(X_{T_0})$ is nonempty. 
\end{proof}

\begin{eg} \label{eg:igusa}
We mention that Proposition~\ref{smoothness} need not be true if $X \to S$ has
fibers of higher dimension. If $S = \spec k$ is the spectrum of an algebraically
closed field of positive characteristic, \cite{Igusa} shows that there exists a 
smooth surface $X$ over $k$ such that $\Pic_{X/k}$ is not reduced, and hence not smooth.
Corollary~\ref{Pic-coarse} gives a smooth surjective morphism
of algebraic stacks $\Bun_1 \to \Pic_{X/S}$. Since the property of smoothness
is local on the source in the smooth topology \cite[Th\'eor\`eme 17.11.1]{EGA4d}, 
we conclude that $\Bun_1$ is not smooth over $S$. 
\end{eg}


\newcommand{\etalchar}[1]{$^{#1}$}
\providecommand{\bysame}{\leavevmode\hbox to3em{\hrulefill}\thinspace}
\providecommand{\MR}{\relax\ifhmode\unskip\space\fi MR }
\providecommand{\MRhref}[2]{%
  \href{http://www.ams.org/mathscinet-getitem?mr=#1}{#2}
}
\providecommand{\href}[2]{#2}

\end{document}